\documentclass[a4paper,reqno]{amsart}

\usepackage[utf8]{inputenc}
\usepackage{amssymb}
\usepackage{amsmath}
\usepackage{hyperref}
\usepackage{tikz-cd}
\usepackage{mathtools}
\usepackage[all]{xy}
\usepackage{enumitem}
\usepackage{aliascnt} 

\hypersetup{
    colorlinks,
    linkcolor={red!50!black},
    citecolor={blue!50!black},
    urlcolor={blue!80!black}
}

\newtheorem{theorem}{Theorem}[section]

 \newaliascnt{lemma}{theorem}
  \newtheorem{lemma}[lemma]{Lemma}
  \aliascntresetthe{lemma}

  \newaliascnt{proposition}{theorem}
  \newtheorem{proposition}[proposition]{Proposition}
  \aliascntresetthe{proposition}

  \newaliascnt{corollary}{theorem}
  \newtheorem{corollary}[corollary]{Corollary}
  \aliascntresetthe{corollary}

 \newaliascnt{definition}{theorem}
\newtheorem{definition}[definition]{Definition}
\aliascntresetthe{definition}

\theoremstyle{definition}

  \newaliascnt{remark}{theorem}
  \newtheorem{remark}[remark]{Remark}
  \aliascntresetthe{remark}

\newcommand{\eps}{\varepsilon} 

\renewcommand{\_}{\underline{\,\,\,\,}}

\newcommand{\CC}{\ensuremath{\mathbb{C}}}

\newcommand{\PP}{\ensuremath{\mathbb{P}}}
 
\newcommand{\RR}{\ensuremath{\mathbb{R}}} 

\newcommand{\ZZ}{\ensuremath{\mathbb{Z}}}

\newcommand{\cF}{\mathcal{F}}

\newcommand{\cL}{\mathcal{L}}   
\newcommand{\cM}{\mathcal{M}}

\newcommand{\cO}{\mathcal{O}}

\newcommand{\cZ}{\mathcal{Z}}
\newcommand{\reg}{\mathcal{O}}
\newcommand{\sym}{\mathfrak{S}}
\newcommand{\Ho}{\mathrm{H}}

\DeclareMathOperator{\rep}{rep}

\DeclareMathOperator{\Amp}{Amp}

\DeclareMathOperator{\Bl}{Bl}

\DeclareMathOperator{\ch}{ch}

\DeclareMathOperator{\Coh}{Coh}

\DeclareMathOperator{\Conv}{Conv}

\DeclareMathOperator{\Ext}{Ext}
\DeclareMathOperator{\ext}{ext}

\DeclareMathOperator{\Hom}{Hom}

\DeclareMathOperator{\id}{id}

\DeclareMathOperator{\Mov}{Mov}

\DeclareMathOperator{\NS}{NS}

\DeclareMathOperator{\Pic}{Pic}

\DeclareMathOperator{\rk}{rk}

\DeclareMathOperator{\Spec}{Spec}
\DeclareMathOperator{\Ind}{Ind}
\DeclareMathOperator{\Res}{Res}

\DeclareMathOperator{\Stab}{stab}

\DeclareMathOperator{\pr}{pr}

\DeclareMathOperator{\SStab}{sstab}
\DeclareMathOperator{\FM}{FM}

\newcommand{\KA}{K_n(A)}
\newcommand{\SA}{S_n(A)}

\newcommand{\PA}{P_n(A)}
\newcommand{\hA}{\widehat{A}}
\newcommand{\bA}{\overline{A}}
\newcommand{\vext}{\overrightarrow{\mathrm{ext}}}

\newcommand{\leqp}{%
	\mathrel{\raisebox{-0.5ex}{$\scriptscriptstyle($}}%
	\leq
	\mathrel{\raisebox{-0.5ex}{$\scriptscriptstyle)$}}%
}

\title[Non-symplectic moduli of sheaves on hyperkähler]{A smooth but non-symplectic moduli of sheaves on a hyperk\"ahler variety}

\author{Andreas Krug}
\address{Institut f\"ur Algebraische Geometrie, Leibniz Universit\"at Hannover, Welfengarten 1, 30167 Hannover, Germany}
\email{krug@math.uni-hannover.de}

\author{Fabian Reede}
\address{Institut f\"ur Algebraische Geometrie, Leibniz Universit\"at Hannover, Welfengarten 1, 30167 Hannover, Germany}
\email{reede@math.uni-hannover.de}

\author{Ziyu Zhang}
\address{Institute of Mathematical Sciences, ShanghaiTech University, 393 Middle Huaxia Road, 201210 Shanghai, P.R.China}
\email{zhangziyu@shanghaitech.edu.cn}


\begin{document}

\begin{abstract}
For an abelian surface $A$, we consider stable vector bundles on a generalized Kummer variety $K_n(A)$ with $n>1$. We prove that the connected component of the moduli space which contains the tautological bundles associated to line bundles of degree $0$ is isomorphic to the blowup of the dual abelian surface in one point. We believe that this is the first explicit example of a component which is smooth with a non-trivial canonical bundle.
\end{abstract}

\maketitle

\section*{Introduction}

\subsection{Background and motivation}

Compact hyperk\"ahler manifolds are one of the fundamental building blocks of compact manifolds with trivial first Chern classes. Much effort has been devoted to construct examples of them, including \cite{beauville_83,mukai_86,yoshioka_01}. In particular, many of these examples are constructed by means of moduli spaces of (semi)stable sheaves on projective surfaces with trivial canonical bundles, which are K3 and abelian surfaces.

Let $S$ be a polarized K3 or abelian surface. We fix a Mukai vector $v$ and assume that $E$ is a stable sheaf on $S$. Then $E$ represents a closed point in the moduli space $\cM$ of stable sheaves of $S$. It is known that the tangent space of $\cM$ at the point $[E] \in \cM$ represented by $E$ is given as
$$ T_{[E]}\cM = \Ext_S^1(E, E). $$
Since the canonical bundle $\omega_S \cong \cO_S$ and $\dim S = 2$, the Serre duality gives
$$ \Ext_S^1(E, E) \cong \Ext_S^1(E, E \otimes \omega_S)^\vee \cong \Ext_S^1(E, E)^\vee, $$
which gives a non-degenerate pairing (i.e. the Yoneda pairing) on $\Ext_S^1(E, E)$. Moreover, it is clear that the pairing is anti-symmetric. It follows that the tangent space $\Ext_S^1(E, E)$ is naturally a symplectic vector space. If we assume that every closed point of $\cM$ is represented by a stable sheaf, then globalizing the above analysis shows that $\cM$ carries a holomorphic symplectic structure. 

It is certainly worth wondering whether the moduli space $\cM$ still carries a symplectic structure if the underlying surface $S$ itself is replaced by a hyperk\"ahler variety of higher dimension. However, in such a case the Serre duality would not give a non-degenerate pairing on $\Ext_S^1(E, E)$, hence there is no reason for us to expect that $\cM$ carries a symplectic structure. Nevertheless, many people have studied various examples of these moduli spaces; see for example \cite{schlick_10,wandel_kummer,stapleton_taut_2016,reede_zhang_22}. In recent years, the notions of ``modular sheaves'' and ``atomic sheaves'' have been developed, and moduli spaces of these sheaves have become of particular interests; see for example \cite{ogrady_22,markman_24,fati_24,beckmann_24,bottini_24,guo_24}. The reason is that for special choices of the Mukai vector $v$, it is still possible that the moduli space $\cM$ is hyperk\"ahler. This idea of constructing hyperk\"ahler manifolds has shown to be popular and successful.

The main motivation of our work is exactly the other side of the dichotomy. For a random choice of the Mukai vector $v$, even if the moduli space $\cM$ is smooth, there is a priori no reason to expect $\cM$ to be symplectic, or at least K-trivial. However, to the best of our knowledge, we have not seen any example in the literature, where such a smooth moduli space was explicitly constructed and shown not to have a trivial canonical bundle. We think it is a missing piece in the jigsaw puzzle. Indeed, there is an even deeper motivation of our work, that is to study general properties of moduli spaces of stable sheaves on hyperk\"ahler varieties of higher dimension.

From the existing literature, we have not seen much work in this direction. We suspect that the main difficulty is to determine the stability of sheaves on varieties of higher dimension, and to find enough stable sheaves occupying a complete component of their moduli spaces. In this manuscript, we will construct the first example of such a complete component.

\subsection{Approach and main result}

Our work is inspired by the earlier studies on the slope stability of tautological bundles with respect to some suitable polarization. Stapleton proved the following result on the stability of tautological bundles on Hilbert schemes
\begin{theorem}[{\cite[Theorem A]{stapleton_taut_2016}}]
	Let $S$ be a smooth projective surface, and $E$ a slope stable vector bundle on $S$, not isomorphic to $\cO_S$. Then the tautological bundle $E^{[n]}$ is a slope stable vector bundle on the Hilbert scheme $S^{[n]}$.
\end{theorem}
A family version of the above result was also established
\begin{theorem}[{\cite[Theorem 1.5]{reede_descent_2024}}]
	Let $S$ be a projective K3 surface, and $E$ a slope stable vector bundle on $S$, such that $$\ch(E) \neq \ch(\cO_S).$$ Assume that $\cM_S(\ch(E))$ is a smooth projective moduli space of slope vector bundles on $S$. Then the moduli space $\cM_{S^{[n]}}(\ch(E^{[n]}))$ of slope stable vector bundles on $S^{[n]}$ contains a smooth connected component isomorphic to $\cM_S(\ch(E))$. In other words, there is an inclusion
	$$ \begin{tikzcd}
 		\cM_S(\ch(E)) \ar[r, "f"] & \cM_{S^{[n]}}(\ch(E^{[n]}))
 	\end{tikzcd} $$
	which embeds $\cM_S(\ch(E))$ as a smooth connected component of $\cM_{S^{[n]}}(\ch(E^{[n]}))$. 
\end{theorem}
The above theorem shows that every smooth moduli space of slope stable vector bundles on a projective K3 surface $S$, as long as it does not parametrize $\cO_S$, can be embedded as a smooth connected component of the moduli space of slope stable vector bundles on the Hilbert scheme $S^{[n]}$ via the tautological construction.

The above pair of results can also be extended to generalized Kummer varieties
\begin{theorem}[{\cite[Proposition 2.9]{reede_stable_2022}}]
	Let $A$ be an abelian surface, $n \geq 2$, and $E$ a slope stable vector bundle on $S$, not isomorphic to $\cO_S$. Then the tautological bundle $E^{(n)}$ is a slope stable vector bundle on the generalized Kummer variety $\KA$.
\end{theorem}
Similarly, we have also obtained
\begin{theorem}[{\cite[Theorem 2.10]{reede_stable_2022}}]
	Let $A$ be an abelian surface, $n \geq 2$, and $E$ a slope stable vector bundle on $A$, such that $$\ch(E) \neq \ch(\cO_A).$$ Assume that $\cM_A(\ch(E))$ is a smooth projective moduli space of slope vector bundles on $A$. Then the moduli space $\cM_{\KA}(\ch(E^{(n)}))$ of slope stable vector bundles on $\KA$ contains a smooth connected component isomorphic to $\cM_A(\ch(E))$. In other words, there is an inclusion
	$$ \begin{tikzcd}
 		\cM_A(\ch(E)) \ar[r, "f"] & \cM_{\KA}(\ch(E^{(n)}))
 	\end{tikzcd} $$
	which embeds $\cM_A(\ch(E))$ as a smooth connected component of $\cM_{\KA}(\ch(E^{(n)}))$. 
\end{theorem}
In the exceptional case when $\ch(E) = \ch(\cO_A)$, we observe that
$$ \cM_A(\ch(E)) = \Pic^0(A) = \hA. $$
In such a case, if $E$ is a non-trivial line bundle of degree $0$ on $A$, then $E^{(n)}$ is still a slope stable vector bundle. However, we will see that $\cO_A^{(n)}$ is unstable, which implies that the rational map
	\begin{equation}\label{eqn:trivial-split}
		\begin{tikzcd}
 			\hA \ar[r,dashed,"f"] & \cM_{\KA}(\ch(\cO_A^{(n)}))
 		\end{tikzcd}
	\end{equation}
 	that sends a line bundle $L$ to its tautological bundle $L^{(n)}$ is undefined at a single point representing the trivial line bundle $\cO_A$. Since both moduli spaces are projective, a natural question is to resolve this rational map. The answer to this question is the main result of the present paper
\begin{theorem}[\autoref{thm:sec2}, \autoref{lem:negative-result}, \autoref{prop:main-result}]\label{thm:main}
 	Let $A$ be an abelian surface and $n \geq 2$.
 	\begin{enumerate}[label=(\roman*)]
 		\item The tautological bundle $\cO_A^{(n)}$ splits as
	 	$$ \cO_A^{(n)} = \cO_{\KA} \oplus Q $$
	 	such that both direct summands are stable of different slopes. As a result $\cO_A^{(n)}$ is unstable.
   \item The extension group satisfies $$ \dim \Ext^1(\cO_{\KA}, Q) = 2, $$
	 	and every non-trivial extension of $\cO_{\KA}$ by $Q$ is stable.
 		\item The rational map \eqref{eqn:trivial-split} can be resolved by blowing up $\hA$ at the neutral point $o \in \hA$ representing the trivial bundle $\cO_A$
 	\begin{equation*}
 		\begin{tikzcd}
 			\Bl_{o} \hA \ar[d] \ar[dr,"\overline{f}"] & \\
 			\hA \ar[r,dashed,"f"] & \cM_{\KA}(\ch(\cO_A^{(n)}))
 		\end{tikzcd}
 	\end{equation*}
 	such that the exceptional divisor of the blowup parametrizes the universal extension of $\cO_{\KA}$ by $Q$.
 	\item The classifying morphism $\overline{f}$ embeds $\Bl_{o} \hA$ as a smooth connected component of $\cM_{\KA}(\ch(\cO_A^{(n)}))$. In particular, we constructed a smooth projective connected component of the moduli space of stable bundles on $\KA$ with a non-trivial canonical bundle.
 	\end{enumerate}
\end{theorem}

 As far as we know, this is the first explicit example of such a component of moduli spaces of stable bundles on hyperk\"ahler varieties of arbitrary even dimension $2n \geq 4$.
 
In this paper we always assume that $n \geq 2$. In the case of $n=1$, the generalized Kummer variety $K_1(A)$ is the Kummer K3 surface associated to the abelian surface $A$. In this a case Wandel considered the tautological bundles associated to line bundles of degree $0$ in \cite[Example 3.8]{wandel_kummer}, and concluded that the moduli space of stable tautological bundles is precisely the Kummer surface associated to the dual abelian surface $\hA$. In other words, the rational map 
 	\begin{equation*}
 		\begin{tikzcd}
 			\hA \ar[r,dashed] & \cM_{K_1(A)}(\ch(\cO_A^{(1)}))
 		\end{tikzcd}
 	\end{equation*}
is generically a $2$-to-$1$ cover, which can be resolved by blowing up $\hA$ at the $16$ $2$-torsion points.
	
We could also ask whether a similar example can be constructed on the Hilbert scheme $S^{[n]}$ of a K3 surface $S$. Indeed, the tautological bundle $\cO_S^{[n]}$ associated to $\cO_S$ also splits as a direct sum of two stable bundles on $S^{[n]}$
$$ \cO_S^{[n]} = \cO_{S^{[n]}} \oplus Q. $$
However, unlike the case of abelian surfaces, the structure sheaf $\cO_S$ on a K3 surface $S$ is a rigid bundle, and we can find out that $\Ext^1(\cO_{S^{[n]}}, Q) = 0$, hence we cannot obtain a component of the moduli space of stable bundles on $S^{[n]}$ by any similar construction.

\subsection{Structure of the paper}

The main contents of the paper are organized in three sections. In \autoref{sec:extension}, we recall the construction of tautological bundles and compute the dimensions of some relevant extension groups, which will be used later for the dimension of various strata of moduli spaces.
The computations in \autoref{sec:extension} heavily use the derived McKay correspondence and hence equivariant sheaves.

In \autoref{sec:stability}, we establish the stability of the non-trivial extensions of $\cO_{\KA}$ by $Q$ as mentioned in \autoref{thm:main}.

Finally in \autoref{sec:component}, we apply an elementary transformation to construct a universal family of stable bundles parametrized by $\Bl_{o}\hA$, and show that it is a smooth connected component of the moduli space of stable bundles on $\KA$.

All objects in this text are defined over the field of complex numbers $\CC$.

\subsection*{Acknowledgement}

Z.~Zhang is supported by National Natural Science Foundation of China (Grant No.~12371046) and Science and Technology Commission of Shanghai Municipality (Grant No.~22JC1402700).

\section{Computation of extension groups}\label{sec:extension}

\subsection{Tautological bundles on Kummer varieties}
From now on, let $A$ always be an abelian surface.
The main object of interest in this paper is the tautological bundle $\cO_A^{(n)}$ on the generalized Kummer variety $\KA$ associated to the structure sheaf $\cO_A$ on the abelian surface $A$. We follow the notation in \cite{reede_stable_2022} and recall some of the most important facts used in the rest of this paper in this section.
 
\begin{definition}
For $n \geqslant 1$ we define the generalized Kummer variety $\KA$ by
\begin{equation*}
	\KA\coloneqq \mathrm{alb}^{-1}(0)\subset A^{[n+1]}
\end{equation*}
where $\mathrm{alb}: A^{[n+1]}\rightarrow A$ is the Albanese morphism of the Hilbert scheme of $n+1$-points in $A$.
\end{definition}

\begin{remark}
	By the results in \cite{beauville_83} the variety $\KA$ is a $2n$-dimensional hyperk\"ahler variety. The surface $K_1(A)$ is the classical Kummer K3 surface.
\end{remark}
 
 \begin{remark}
 It is well known that $\mathrm{alb}=m\circ \mathsf{HC}$ where 
 \[
m\colon A^{(n+1)}\rightarrow A\quad,\quad (a_1,\dots,a_{n+1})\mapsto \sum_{i=1}^{n+1}a_i 
 \]  is the addition morphism on the symmetric power and $\mathsf{HC}: A^{[n+1]}\rightarrow A^{(n+1)}$ is the Hilbert-Chow morphism; see \cite[Section 4.3.1]{yoshioka_01}.
 \end{remark}

The generalized Kummer variety $\KA$ comes with a universal family $\cZ \subset A \times \KA$ of length $n+1$ subschemes of $A$ whose weighted support sums up to $0$.
It has two natural projections 
\begin{equation}
	\label{projections_universal}
 \begin{tikzcd}
		A & \cZ \arrow[l, swap, "p"] \arrow[r,"q"] & \KA.
	\end{tikzcd}
\end{equation}

\begin{definition}
Let $E$ be a vector bundle on $A$. The associated tautological vector bundle on $\KA$ is given as
\begin{equation*}
	E^{(n)}\coloneqq q_{*}p^{*}E.
\end{equation*}
\end{definition}

\begin{remark}
	As $q:\cZ\rightarrow \KA$ is finite and flat of degree $n+1$ the (a priori only) coherent sheaf $E^{(n)}$ is indeed locally free on $\KA$ and we have
	\begin{equation*}
		\rk(E^{(n)})=(n+1)\rk(E).
	\end{equation*}
\end{remark}

The main goal of this section is to prove

\begin{theorem}\label{thm:sec2}
 The tautological bundle associated to $\cO_A$ decomposes as
	\begin{equation}\label{eqn:taut-split}
		\cO_A^{(n)}=\cO_{\KA}\oplus Q
	\end{equation}
	for some vector bundle $Q$ on $\KA$ with $\rk(Q)=n$. We have $\ext^1(\reg_{\KA}, Q)=h^1(Q)=2$. Hence, there is a two dimensional family of extensions of $\reg_{\KA}$ by $Q$. Let us fix one such non-trivial extension
 \begin{equation}\label{eq:Fses}
 0\to Q\to F\to \reg_{\KA}\to 0\,.
 \end{equation}
 Then $\ext^1(F,F)=2$.
\end{theorem}

We will heavily use the derived McKay correspondence and equivariant methods, so we start with some preliminaries on this. Note, however, that the splitting \eqref{eqn:taut-split} can easily be proved without the McKay correspondence using that $\reg_A^{(n)}\cong q_*\reg_\cZ$ is the push forward of the structure sheaf of the universal family along the flat and finite morphism $q$ together with \cite[Proposition 5.7]{kollar_birational_1998}.

\subsection{Equivariant sheaves and categories}

For details on equivariant sheaves and their derived categories and functors, we refer to \cite[Section 4]{BKR}, \cite{Elagin--onequi}, \cite[Section 2.2]{krug_remarks_2018}, and \cite{BO--equivariant}. Here, we only sketch the basic definitions and collect some facts that we need later. 

Let $G$ be a finite group acting on a smooth variety $X$. A \emph{$G$-equivariant sheaf} on $X$ is a coherent sheaf $E\in \Coh(X)$ together with a \emph{$G$-linearization}, i.e.\ a family $\{\lambda_g\colon E\xrightarrow\sim g^*E\}_{g\in G}$ of isomorphims satisfying that, for every $g,h\in G$, the composition 
\[
E\xrightarrow{\lambda_g}g^*E\xrightarrow{g^*\lambda_h}g^*h^*E\cong (gh)^*E
\]
equals $\lambda_{gh}$. In particular, $\lambda_e=\id_E$ where $e\in G$ is the neutral element. The $G$-equivariant sheaves form an abelian category, denoted by $\Coh_G(X)$. Its bounded derived category $D^b_G(X)\coloneqq D^b(\Coh_G(X))$ is called the \emph{equivariant (bounded) derived category of $X$}.

Let $H\le G$ be a subgroup. There is the \emph{restriction} functor $\Res_G^H\colon \Coh_G(X)\to \Coh_H(X)$, $(E,\lambda)\mapsto (E,\lambda_{\mid H})$ which restricts the $G$-linearization of an equivariant sheaf to an $H$-linearization. In the special case $H=1$, the restriction functor is isomorphic to the forgetful functor which drops the linearization completely.
There is the \emph{induction} functor $\Ind_H^G\colon \Coh_H(X)\to \Coh_G(X)$ which is left and right adjoint to the restriction functor:
\[
\Res_G^H\vdash \Ind_H^G\vdash \Res_G^H\,.
\]
On the underlying sheaves, it is given by 
\begin{equation}\label{eq:Ind}
\Ind_H^G(E)=\bigoplus_{[g]\in G/H} g^*E
\end{equation}
where the sum runs through some set of representatives of the cosets. The $G$-linearization of $\Ind_H^G(E)$ is then given by a combination of the $H$-linearization of $E$ and permutation of the direct summands in \eqref{eq:Ind}; see e.g.\ \cite[Section 3.2]{BO--equivariant} for details. 
The restriction and the induction functor are both exact and hence also induce an adjoint pair of functors between the equivariant derived categories. 

For two $G$-equivariant sheaves $E,F\in \Coh_G(X)$, their linearizations induce, for every $i\in \ZZ$, an action on the extension space 
\[
\Ext^i(E,F)\coloneqq \Ext^i\bigl(\Res_G^1 E, \Res_G^1 F\bigr) \cong \Hom_{D^b(X)} (\Res_G^1 E, \Res_G^1 F[i]\bigr)
\]
The homomorphisms in the equivariant derived category are then given by the invariants under this action:
\begin{equation}\label{eq:Hominva}
\Hom_{D^b_G(X)}(E,F[i])\cong \Ext^i(E,F)^G=:\Ext^i_G(E,F)\,.
\end{equation}
The structure sheaf $\reg_X$ has a canonical $G$-linearization given by push-forward of regular functions along the group action on $X$. Taking $E=\reg_X$ in \eqref{eq:Hominva} gives
\begin{equation}\label{eq:Hinva}
\Hom_{D^b_G(X)}(\reg_X,F[i])\cong \Ext^i(\reg_X,F)^G\cong\Ho^i(F)^G=:\Ho^i_G(F)\,.
\end{equation}

The tensor product of two $G$-equivariant sheaves on $X$ has a naturally induced $G$-linearization, which means that there is a functor
\[
\_\otimes \_\colon \Coh_G(X)\times \Coh_G(X)\to \Coh_G(X)\,.
\]
Let $f\colon X\to Y$ be a proper $G$-equivariant morphism. Pull-backs and push-forwards of $G$-equivariant bundles under $f$ then have a naturally induced $G$-linearization, which means that we have functors
\[
f_*\colon \Coh_G(X)\to \Coh_G(Y)\quad,\quad f^*\colon \Coh_G(Y)\to \Coh_G(X)\,.
\]
All these functors can be derived to give functors between the equivariant derived categories.

The group $G$ acts trivially on the point $\Spec \CC$. The category $\Coh_G(\Spec\CC)$ is equivalent to the category $\rep(G)$ of finite-dimensional $G$ representation. Hence, for any 
$E\in \Coh_G(X)$ and $\rho\in \rep(G)$, we can define a tensor product
\[
E\otimes \rho\coloneqq E\otimes f^*\rho
\]
where $f\colon X\to \Spec \CC$.

\subsection{The derived McKay correspondence for the generalized Kummer variety}

To explain the derived McKay correspondence for the generalized Kummer variety and the image of tautological bundles under it, let us fix some notation. We write $\sym\coloneqq \sym_{n+1}$ for the permutation group of the numbers $1,\dots,n+1$. It acts by permutation of the factors on $A^{n+1}$. We consider the subvariety 
\[
P\coloneqq P_n\coloneqq \PA\coloneqq \bigl\{ (a_1,\dots, a_{n+1})\mid \sum_{i=1}^{n+1} a_i=0\bigr\}\subset A^{n+1}\,.
\]
This subvariety is preserved by the $\sym$-action on $A^{n+1}$. Hence, by restriction, we get an induced $\sym$-action on $P$. Note that there is a (non $\sym$-equivariant) isomorphism  
\begin{equation}\label{eq:Piso} h\colon A^n\xrightarrow\cong P\quad,\quad(a_1,\dots,a_n)\mapsto (-a_1,\dots, -a_n, \sum_{i=1}^n a_i)\,.\end{equation}
For $i=1,\dots, n+1$, we denote restriction of the projection $p_i\colon A^{n+1}\to A$ to the $i$-th factor to the subvariety $P\subset A^{n+1}$ by
\[
\overline p_i\coloneqq p_{i\mid P}\colon P\to A\,.
\]

\begin{proposition}\label{prop:tautMcKay}
There is an equivalence $\Psi\colon D_\sym(P)\xrightarrow\cong D(\KA)$ satisfying 
$\Psi(\reg_{P})\cong \reg_{\KA}$ and 
\[
\Psi(E^{\{n\}})\cong E^{(n)}\quad\text{where}\quad   E^{\{n\}}\coloneqq \Ind_{\mathfrak{S}_n}^{\mathfrak{S}_{n+1}}(\overline{p}_{n+1}^{*}E)
\]
\end{proposition}

\begin{proof}
This is \cite[Remark 3.14]{krug_remarks_2018}. Note that, in the notation of \emph{loc.\ cit.}, $\Psi$ is $\Psi_K$, $P_nA$ is $N_nA$, $E^{\{n\}}$ is $j^*\mathsf C(E)$, and $\reg_P$ is $j^*\mathsf W^0(\reg_A)$.    
\end{proof}

Indeed, for our purposes, the equivalence $\Psi$ of \cite{krug_remarks_2018} is better suited than the original derived McKay correspondence of \cite{BKR} and \cite{Hai}.

\begin{remark}\label{rem:Psicoh}
Note that, for every object $N\in D_\sym(P)$, we have $\Ho^*_\sym(N)\cong \Ho^*(\Psi(N))$. The reason is that
\begin{align*}
\Ho^i_\sym(N) &\cong \Hom_{D^b_\sym(P)}\bigl(\reg_P,N[i] \bigr)\tag{by \eqref{eq:Hinva}}\\     
&\cong \Hom_{D^b(\KA)}\bigl(\Psi(\reg_P),\Psi(N)[i]\bigr)\tag{$\Psi$ is an equivalence}\\
&\cong \Hom_{D^b(\KA)}\bigl(\reg_{\KA},\Psi(N)[i]\bigr)\tag{$\Psi(\reg_P)\cong \reg_{\KA}$}\\
&\cong \Ho^i(\Psi(N))\,.
\end{align*}
\end{remark}

\begin{proposition}[{compare \cite[Theorem 6.7]{meachan_derived_2015}}]\label{prop:Mea}
We have
\begin{equation}
\mathrm{H}^{*}_{\sym}(\reg_{P})\cong \Ho^*(\reg_{\KA})\cong \CC\oplus \CC[-2]\oplus \dots\oplus \CC[-2n]\cong \frac{\CC[\alpha]}{(\alpha^{n+1})}\,, \quad \deg(\alpha)=2\,.\label{eq:OH}    
\end{equation}
Furthermore, for any coherent sheaf $E$ on $A$, we have
\begin{align}
&\Ho^*_{\sym}(E^{\{n\}})\cong \Ho^*(E)\otimes \Ho^*_{\sym_n}(\reg_{P_{n-1}})\,, \label{eq:tautH}\\
&\Ext^*_{\sym}(E^{\{n\}}, \reg_P)\cong \Ho^*(E^\vee)\otimes \Ho^*_{\sym_n}(\reg_{P_{n-1}})\,, \label{eq:tautHdual}\\
&\Ext^*_{\sym}(E^{\{n\}}, F^{\{n\}})\cong \begin{matrix}\bigl(\Ext^*(E,F)\otimes \Ho^*_{\sym_n}(\reg_{P_{n-1}})\bigr)\\\oplus \bigl(\Ho^*(E^\vee)\otimes \Ho^*(F)\otimes \Ho^*_{\sym_{n-1}}(\reg_{P_{n-2}})\bigr)\end{matrix}\,.\label{eq:tautExt}
\end{align}    
\end{proposition}

\begin{proof}
Formula \eqref{eq:OH} follows from $\Psi(\reg_P)\cong  \reg_{\KA}$ together with the fact that $\KA$ is hyperkähler. Note that one can also compute the equivariant cohomology $\mathrm{H}^{*}_{\sym}(\reg_{P})$ directly, without using the equivalence $\Psi$, by combining \autoref{lem:repdescription}(1) below with \cite[Lemma B.6(1)]{Scala--Coh}.

The existence of the formulas \eqref{eq:tautH}, \eqref{eq:tautHdual}, and \eqref{eq:tautExt} is alluded to in \cite[Remmark 4.4]{krug_remarks_2018}. 
Closely related formulas are provided in \cite[Theorem 6.7]{meachan_derived_2015}. However, none of both references gives exactly the formulas we need. Hence, let us sketch the proofs.  
By the adjunction $\Res_{\sym_{n+1}}^{\sym_n}\vdash \Ind_{\sym_n}^{\sym_{n+1}}$, we have
\begin{align*}
\Ho^*_\sym(E^{\{n\}})\cong \Ext^*_\sym(\reg_P, \Ind_{\sym_n}^{\sym_{n+1}} \overline p_{n+1}^* E)&\cong 
\Ho^*_{\sym}(\Ind_{\sym_n}^{\sym_{n+1}} \overline p_{n+1}^* E)\\
&\cong 
\Ho^*_{\sym_{n}}(\overline p_{n+1}^* E)\,.
\end{align*}
Now note that the isomorphism $h\colon A^n\xrightarrow\cong P$ defined in \eqref{eq:Piso}, while not being $\sym_{n+1}$-equivariant, is $\sym_n$-invariant. Furthermore, $\overline p_{n+1}\circ h=\Sigma_n$ where
\[
\Sigma_n\colon A^n\to A\quad,\quad (a_1,\dots,a_n)\mapsto \sum_{i=1}^n a_i
\]
is the summation map. Hence, we can further compute
\[
\Ho^*_{\sym_{n}}(P,\overline p_{n+1}^* E)\cong \Ho^*_{\sym_{n}}(A^n,\Sigma_n^* E)\cong \Ho^*(A,R\Sigma_{n*}^{\sym_n}\Sigma_n^* E)
\]
where $\Sigma_{n*}^{\sym_n}=(\_)^{\sym_n}\circ \Sigma_{n*}$ denotes the push-forward along the $\sym_n$-invariant morphism $\Sigma_n$ followed by the functor of taking invariants, and $R\Sigma_{n*}^{\sym_n}\colon D^b_{\sym_n}(A^n)\to D^b(A)$ is its right-derived functor.
Now, \eqref{eq:tautH} follows from the isomorphism 
\[
R\Sigma_{n*}\Sigma_n^* E\cong E\otimes\Ho^*(\reg_{K_{n-1}A})\cong E\otimes \Ho^*_{\sym_n}(\reg_{P_{n-1}})\,;
\]
see \cite[Corollary 6.5]{meachan_derived_2015} (the essential reason for this isomorphism is that the fibres of $\Sigma_n$ are all $\sym_n$-equivariantly isomorphic to $P_{n-1}$).

The proof of \eqref{eq:tautHdual} is very similar, using the adjunction $\Ind_{\sym_n}^{\sym_{n+1}} \vdash \Res_{\sym_{n+1}}^{\sym_n}$ in the other direction.

For \eqref{eq:tautExt}, we first note that 
\[
\Ext^*_\sym(E^{\{n\}},F^{\{n\}})\cong \Ext^*(E^{\{n\}},F^{\{n\}})^{\sym}\cong\bigl[ \bigoplus_{i,j=1}^{n+1} \Ext^*(\overline p_i E, \overline p_j F)\bigr]^\sym
\]
By what is known as Danila's Lemma (see e.g.\ \cite[Lemma 2.2]{Dan}, \cite[Remark 2.4.2]{Scala--Coh}), the $\sym$-invariants of the double sum can be computed as
\begin{equation}\label{eq:Danila}
\bigl[ \bigoplus_{i,j=1}^{n+1} \Ext^*(\overline p_i^* E, \overline p_j^* F)\bigr]^\sym
\cong 
\Ext^*(\overline p_{n+1}^* E, \overline p_{n+1}^* F)^{\sym_n}\oplus \Ext^*(\overline p_{n+1}^* E, \overline p_{n}^* F)^{\sym_{n-1}}\,;    
\end{equation}
compare e.g.\ the proof of \cite[Proposition 4.1]{krug_remarks_2018}. Now, computations similar to those performed in the proof of \eqref{eq:tautH} above identify the two direct summands of the right side of \eqref{eq:Danila} with the two direct summands of the right side of \eqref{eq:tautExt} (for the second summand, one uses \cite[Corollary 6.8]{meachan_derived_2015} in place of \cite[Corollary 6.5]{meachan_derived_2015}).
\end{proof}

\subsection{The splitting of the tautological bundle}

Let $R$ be the \emph{permutation representation} of $\sym=\sym_{n+1}$, i.e.\ the $n+1$-dimensional representation $R=\CC^{n+1}$ such that $\sym$ acts on the canonical basis $\eps_1,\dots,\eps_{n+1}$ by 
$\sigma\cdot \eps_i=\eps_{\sigma(i)}$. It splits as 
\begin{equation}\label{eq:Psplit}
R\cong \CC\oplus \rho
\end{equation}
where $\CC$ stands for the one-dimensional trivial representation (embedded diagonally into $R$ as the invariant subspace), and $\rho$ is the $n$-dimensional \emph{standard representation}. The latter can be realised as the subspace of $P$ of vectors $v=(v_1,\dots,v_{n+1})$ of $P$ with $\sum v_i=0$. Equivalently, and this is the definition that we will use in the following, it is the quotient
\[
\rho=R/\CC=\CC^{n+1}/\langle (1,\dots,1)\rangle\,.
\]
Setting $e_i=\overline \eps_i$, we have the basis $e_1,\dots,e_n$ of $\rho$. Note that
\begin{equation}\label{eq:en+1}
e_{n+1}=-\sum_{i=1}^n e_i\,.
\end{equation}

\begin{lemma}\label{lem:split}
We have an isomorphism of $\sym$-equivariant sheaves
\[
(\reg_A)^{\{n\}}\cong \reg_P\otimes R\,.
\]
In particular, writing $\reg^{\{n\}}\coloneqq (\reg_A)^{\{n\}}$ and $G\coloneqq \reg_P\otimes \rho$, we have a splitting
\begin{equation}\label{eq:splitOn}
\reg^{\{n\}}\cong \reg_P\oplus G\,.
\end{equation}
\end{lemma}

\begin{proof}
Note that $R\cong \Ind_{\sym_{n}}^{\sym_{n+1}}\CC$ where $\CC$ is the trivial $\sym_n$-representation.

One way to deduce the isomorphism $(\reg_A)^{\{n\}}\cong \reg_P\otimes R$ is to use the general compatibility $\Ind_H^G\circ f^*\cong f^*\circ \Ind_H^G$, which holds for any finite group $G$, any subgroup $H\le G$, and any $G$-equivariant morphism $f\colon X\to Y$; see e.g.\ \cite[Equation (3)]{krug_remarks_2018}.

Indeed, taking $f\colon P\to \Spec \CC$ the projection to a point, and using the identification of equivariant sheaves on $\Spec\CC$ with representations, we get
\[
(\reg_A)^{\{n\}}=\Ind_{\sym_n}^{\sym_{n+1}}\reg_P\cong \Ind_{\sym_n}^{\sym_{n+1}} f^*\CC\cong f^*\Ind_{\sym_n}^{\sym_{n+1}}\CC \cong f^* R\cong \reg_P\otimes R\,.
\]
The second assertion follows directly from the splitting \eqref{eq:Psplit}.
\end{proof}

Setting $Q\coloneqq \Psi(G)$ and recalling that $\Psi(\reg_P)\cong \reg_{\KA}$ and $\Psi(\reg^{\{n\}})\cong \reg_A^{(n)}$ by \autoref{prop:tautMcKay}, we get the decomposition \eqref{eqn:taut-split} of \autoref{thm:sec2} since the equivalence $\Psi$ preserves direct sums.

Next, we want to compute the cohomology of $Q=\Psi(G)$. By \autoref{rem:Psicoh}, we can compute $\Ho^*_\sym(G)$ instead. To be able to state the result in a compact form, we introduce the following notation for any $\sym$-equivariant coherent sheaf $N$ on $P$:
\[
h^i_\sym(N)\coloneqq \dim \Ho^i_\sym(N) \quad,\quad \vec h_{\sym}(N)\coloneqq \bigl(h^0_{\sym}(N), h^1_{\sym}(N),\dots, h^{2n}_{\sym}(N)\bigr)\in (\mathbb N_0)^{2n}\,.
\]
\begin{lemma}\label{lem:hG} We have
\begin{equation}\label{eq:hG}
\vec h_{\sym}(G)=(0,2,1,2,1,\dots, 2,1,2,0)\,.
\end{equation}    
In particular, 
\[
h^1_{\sym}(G)=\ext^1_{\sym}(\reg_P,G)=h^1(Q)=\ext^1(\reg_{\KA}, Q)=2\,.
\]
\end{lemma}

\begin{proof}
The splitting \eqref{eq:splitOn} gives $\Ho^*_\sym(\reg^{\{n\}})\cong\Ho^*_\sym(\reg_P)\oplus \Ho^*_\sym(G)$. Hence, $\vec h_\sym(G)=\vec h_\sym(\reg^{\{n\}})- \vec h_\sym(\reg_P)$.
By \eqref{eq:tautH}, we have 
\[
\Ho^*(\reg^{\{n\}})\cong \Ho^*(\reg_A)\otimes \Ho^*_\sym(\reg_{P_{n-1}})
\]
which has dimension vector $\vec h_\sym(\reg^{\{n\}})=(1,2,\dots,2,1)$; compare \eqref{eq:OH}. In summary,
\begin{align*}
\vec h_\sym(G)=\vec h_\sym(\reg^{\{n\}})- \vec h_\sym(\reg_P)    
&= (1,2,\dots,2,1)-(1,0,1,\dots,1,0,1)\\
&=(0,2,1,2,1,\dots, 2,1,2,0)\,. \qedhere
\end{align*}
\end{proof}

\subsection{Computing the extensions of the extension}

From now on, we fix some $0\neq \vartheta\in \Ext^1_{\sym}(\reg_P,G)\cong \Ext^1(\reg_{\KA}, Q)$ and consider the corresponding extension 
\begin{equation}\label{eq:Mses}
0\to G\xrightarrow\alpha M\xrightarrow\beta \reg_P\to 0    
\end{equation}
of $\sym$-equivariant vector bundles on $P$. This sequence corresponds under the equivalence $\Psi$ to 
\eqref{eq:Fses}. Hence, in order to finish the proof of \autoref{thm:sec2}, we need to show that 
\[
\ext^1_{\sym}(M,M)=2\,.
\]    
We will prove $\Ext^1_\sym(M,M)\cong \CC^2$ by working through the diagram
\begin{equation}\label{eq:Extlattice}
\xymatrix{
\Ho^*_\sym(G) \ar^{\beta^*}[d] \ar^{\alpha_*}[r]  & \Ho^*_\sym(M)   \ar^{\beta_*}[r] \ar^{\beta^*}[d]   &  \Ho^*_\sym(\reg_P)  \ar^{\beta^*}[d] \ar^{\vartheta_*}[r] &   \\
 \Ext^*_\sym(M,G)  \ar^{\alpha^*}[d]  \ar^{\alpha_*}[r]  &  \Ext^*_\sym(M,M) \ar^{\alpha^*}[d]  \ar^{\beta_*}[r]  &  \Ext^*_\sym(M,\reg_P) \ar^{\alpha^*}[d] \ar^{\qquad\vartheta_*}[r]&   \\
\Ext^*_\sym(G,G)   \ar^{\alpha_*}[r] \ar^{\vartheta^*}[d]  & \Ext^*_\sym(G,M)   \ar^{\beta_*}[r] \ar^{\vartheta^*}[d]  &  \Ext^*_\sym(G,\reg_P) \ar^{\qquad\vartheta_*}[r] \ar^{\vartheta^*}[d]& \\
 & & & 
}
\end{equation}
induced by the sequence \eqref{eq:Mses} defining $M$. All squares in this diagram commute, and all columns and rows of this diagram are exact triangles in the derived category of vector spaces. This means that they correspond to long exact sequences. For example, the middle column corresponds to the long exact sequence
\begin{equation*}
 \begin{tikzcd}
0 \arrow[r]
& \Ho^0_\sym(M) \arrow[r]\arrow[d, phantom, ""{coordinate, name=Z}]
& \Hom_\sym(M,M) \arrow[r]
& \Hom_\sym(M,G) \arrow[dlll,
rounded corners,
to path={ -- ([xshift=2ex]\tikztostart.east)
|- (Z) [near end]\tikztonodes
-| ([xshift=-2ex]\tikztotarget.west)
-- (\tikztotarget)}] \\
\Ho^1_\sym(M) \arrow[r]
& \Ext^1_\sym(M,M) \arrow[r]
& \Ext^1_\sym(M,G) \arrow[r]
& \Ho^2_\sym(M) \arrow[r] & {}
\end{tikzcd}   
\end{equation*}

We already know the dimensions of the graded pieces on the upper corners of \eqref{eq:Extlattice} by \eqref{eq:hG} and \eqref{eq:OH}. We will now also compute the dimensions of the graded pieces of the lower corners. For this purpose, we introduce the notation that, for two $\sym$-equivariant sheaves $C$ and $D$ on $P$, we write 
\[
\vext_\sym(C,D)\coloneqq \bigl(\hom_\sym(C,D), \ext^1_\sym(C,D),\dots, \ext^{2n}_\sym(C,D)\bigr)\,.
\]
\begin{lemma}\label{lem:hextG}We have
\begin{align*}
\vext_\sym(G,\reg_P)&=(0,2,1,2,1,\dots, 2,1,2,0)\,,\\
\vext_\sym(G,G)&=\begin{cases}(1,2,5,2,1)&\quad\text{for $n=2$,}
\\
(1,2,6,6,6,2,1)&\quad\text{for $n=3$,}
\\
(1,2,6,6,7,6,\dots,6,7,6, 6,2,1)&\quad\text{for $n\ge 4$.}
\end{cases}
\end{align*}
\end{lemma}
\begin{proof}
The formula for $\vext_\sym(G,\reg_P)$ follows by \eqref{eq:hG} together with equivariant Serre duality.

To compute $\vext_\sym(G,G)$, first note that the splitting $\reg^{\{n\}}\cong \reg_P\oplus G$ gives
\[
 \Ext^*_\sym(\reg^{\{n\}},\reg^{\{n\}})\cong \Ho^*_\sym(\reg_P)\oplus \Ho^*_\sym(G)\oplus \Ext^*_\sym(G,\reg_P)\oplus \Ext^*_\sym(G,G)\,.   
\]
On the other hand, \eqref{eq:tautExt} and \eqref{eq:tautH} give
\begin{align*}
  \Ext^*_\sym(\reg^{\{n\}},\reg^{\{n\}})&\cong \bigl(\Ho^*(\reg_A)\otimes \Ho^*_{\sym}(\reg_{P_{n-1}})\bigr)\oplus \bigl(\Ho^*(\reg_A)\otimes \Ho^*(\reg_A)\otimes \Ho^*_{\sym}(\reg_{P_{n-2}})\bigr)\\
  &\cong \Ho^*(\reg^{\{n\}})\oplus \bigl(\Ho^*(\reg_A)\otimes \Ho^*(\reg_A)\otimes \Ho^*_{\sym}(\reg_{P_{n-2}})\bigr)
  \,.  
\end{align*}
Noting again that $\Ho^*_\sym(\reg^{\{n\}})\cong \Ho^*_\sym(\reg_P)\oplus \Ho^*_\sym(G)$ we can cancel the direct summands in the two different decompositions of $\Ext^*_\sym(\reg^{\{n\}},\reg^{\{n\}})$ to get a (non-canonical) isomorphism of graded vector spaces
\begin{equation}\label{eq:aftercancel}
\Ext^*_\sym(G,\reg_P)\oplus \Ext^*_\sym(G,G)\cong \Ho^*(\reg_A)\otimes \Ho^*(\reg_A)\otimes \Ho^*_{\sym}(\reg_{P_{n-2}}) 
\end{equation}
The graded vector space $\Ho^*(\reg_A)\otimes \Ho^*(\reg_A)\otimes \Ho^*_{\sym}(\reg_{P_{n-2}})$ has dimension vector (meaning the vector whose $i$th entry is the dimension of the degree $i$ part)
\[
\overrightarrow{d}= \begin{cases}(1,4,6,4,1)&\quad\text{for $n=2$,}
\\
(1,4,7,8,\dots,8, 7,2,1)&\quad\text{for $n\ge 3$.}
\end{cases}
\]
We get the desired result for $\vext_\sym(G,G)$ since we already computed $\vext_\sym(G,\reg_P)$ and 
\eqref{eq:aftercancel} gives
\[
\vext_\sym(G,G)=\overrightarrow{d} - \vext_\sym(G,\reg_P)\,. \qedhere
\]
\end{proof}

Now we know the graded vector spaces at all four corners of \eqref{eq:Extlattice}. So if we understand the connecting degree 1 maps $\vartheta_*$ and $\vartheta^*$ of the diagram, we can compute also the dimensions of the other graded vector spaces, and in particular, the dimension of $\Ext^1_\sym(M,M)$ in the middle of the diagram that we are interested in. As $\vartheta_*$ and $\vartheta^*$ are given by Yoneda products with $\vartheta\in \Ho^1_\sym(G)\cong\Ext^1_\sym(\reg_P,G)$, they are described by the following 

\begin{lemma}\label{lem:repdescription}
We consider the 2-dimensional complex vector space
\[
V\coloneqq \Ho^1(\reg_A)\,.
\]
\begin{enumerate}
 \item There are $\sym$-equivariant isomorphism of algebras
 \[
\Ext^*(\reg_P,\reg_P)\cong \Ho^*(\reg_P)\cong \wedge^*(V\otimes \rho) 
 \]
where the multiplication on the left side is the Yoneda product, the multiplication in the middle is the cup product, and the multiplication on the left is the wedge product.
\item Given two $\sym$-representations $U$ and $W$, we have a canonical $\sym$-equivariant isomorphism 
\[
\Ext^*(\reg_P\otimes U,\reg_P\otimes W)\cong \wedge^*(V\otimes \rho)\otimes \Hom(U,W)
\]
\item Given three $\sym$-representations $U,W, X$ the canonical isomorphisms make the diagram
\begin{equation*}
    \begin{tikzcd}
       \Ext^*(\reg_P\otimes W,\reg_P\otimes X)\times \Ext^*(\reg_P\otimes U,\reg_P\otimes W) \arrow[r, "\mathrm{Yon}"]\arrow[d,"\cong"] & \Ext^*(\reg_P\otimes U,\reg_P\otimes X)  \arrow[d,"\cong"]   \\
 \Bigl(\wedge^*(V\otimes \rho)\otimes \Hom(W,X)\Bigr) \times \Bigl(\wedge^*(V\otimes \rho)\otimes \Hom(U,W)\Bigr)   \arrow[r]  &     \wedge^*(V\otimes \rho)\otimes \Hom(U,X)  
    \end{tikzcd}
\end{equation*}
commute. Here, the upper horizontal map is the Yoneda product, and is the lower horizontal map is the tensor product of the wedge product on $\wedge^*(V\otimes \rho)$ and the composition on $\Hom(P,P)$.
\end{enumerate}    
\end{lemma}

\begin{proof}
The statement of part (1) can be extracted from \cite[Section 3.1]{AndWis--Kummer}, but let us give a proof. 
As $P\cong A^n$ is an abelian variety, we have an isomorphism of algebras $\Ho^*(\reg_P)\cong \wedge^* \Ho^1(\reg_P)$; see \cite[Corollary 1 on p.\ 8]{Mum--abelianbook}.
Hence, it suffices to proof that $\Ho^1(\reg_P)\cong V\otimes \rho$ as representations of $\sym=\sym_{n+1}$.
Considering the permutation action of $\sym$ on $A^{n+1}$, we have, by Künneth formula, an isomorphism of $\sym$-representations $\Ho^1(\reg_{A^{n+1}})\cong V\otimes R$.

The embedding $\iota\colon P\hookrightarrow A^{n+1}$ has the (non-$\sym$-equivariant) retract
\[
A^{n+1}\to P\quad,\quad (a_1,\dots, a_n,a_{n+1})\mapsto (a_1,\dots, a_n,-\sum_{i=1}^n a_i) 
\]
Hence, $\iota^*\colon \Ho^1(\reg_{A^{n+1}})\cong V\otimes R\twoheadrightarrow \Ho^1(\reg_P)$ is surjective. 

By definition of $P$, we have $\Sigma\circ \iota=\Sigma_{\mid P}=0$, where $\Sigma\colon A^{n+1}\to A$ is the addition map. Hence, $\iota^*\circ \Sigma^*\colon V=\Ho^1(\reg_A)\to \Ho^1(\reg_P)$ is the zero map. 
Furthermore, $\Sigma\colon A^{n+1}\to A$ has sections, for example $a\mapsto(a,0,\dots,0)$. Hence,
\[
\Sigma^*\colon V=\Ho^1(\reg_A)\to \Ho^1(\reg_{A^{n+1}})\cong V\otimes R
\]
is injective. Since $\Sigma$ is $\sym$-invariant, the image of $\Sigma^*$ is contained in the space of $\sym$-invariants. As $(V\otimes R)^\sym\cong V\otimes (R^\sym)\cong V$, the dimensions match, and $\Sigma^*\colon V \to \Ho^1(\reg_{A^{n+1}})$ is actually the embedding of the whole space of invariants. Again, because the dimensions match, we can summarize the above to a short exact sequence of $\sym$-representations
\[
0\to V\xrightarrow{\Sigma^*} V\otimes R\xrightarrow{\iota^*} \Ho^1(\reg_P)\to 0
\]
which means that $\Ho^1(\reg_P)$ is the quotient of $V\otimes R$ by the space of invariants. Hence, $\Ho^1(\reg_P)\cong V\otimes \rho$.

The statements (2) and (3) follow from (1) because of the way the tensor product of a vector space with a coherent sheaf is defined.
\end{proof}

Using \autoref{lem:repdescription}(2), we get another explanation of the dimension $h^1_{\sym}(G)=\ext^1_{\sym}(\reg_P,G)=2$ first computed in \autoref{lem:hG}. Namely,
\[
\Ho^1_{\sym}(G)=\Ext_\sym^1(\reg_P,G)\cong\bigl[\wedge^1(V\otimes \rho)\otimes \Hom(\CC,\rho)  \bigr]^\sym\cong \bigl[V\otimes \rho\otimes \rho\bigl]^\sym\cong V\,.
\]
Concretely, the the last isomorphism is given by 
\begin{equation}\label{eq:thetav}
V\xrightarrow \cong \bigl[V\otimes \rho\otimes \rho\bigl]^\sym\quad,\quad v\mapsto \vartheta_v\coloneqq \sum_{i=1}^{n+1} ve_i\otimes e_i\,.
\end{equation}

\begin{remark}
    Here, and in the following, our convention is that we omit the tensor sign for elements of $V\otimes \rho$, but explicitly write down all the other tensor signs.
\end{remark}

We now fix some $v\in V$ such that the class $0\neq \vartheta\in \Ho^1_\sym(G)$ under the isomorphism $\Ho^1(G)\cong \bigl[V\otimes \rho\otimes \rho\bigr]^\sym$ is given by $\vartheta=\vartheta_v$. Furthermore, we fix some $u\in V$ which together with $v$ gives a base:
\[
V=\langle u,v\rangle\,.
\]

\begin{lemma}\label{lem:thetainj1}
The maps $\vartheta_*\colon \Ho^0_{\sym}(\reg_P)\to \Ho^1_{\sym}(G)$ and $\vartheta_*\colon \Ho^2_{\sym}(\reg_P)\to \Ho^3_{\sym}(G)$ are injective.
\end{lemma}

\begin{proof}
We have that $\Ho^0_{\sym}(\reg_P)=\langle 1\rangle$ is spanned by the constant function 1. As $\vartheta_*(1)=\vartheta\neq 0$ , we have the first claimed injectivity.

By \autoref{prop:Mea}, also $\Ho^2_{\sym}(\reg_P)$ is also one-dimensional. Under the isomorphism 
$\Ho^2(\reg_P)\cong \bigl[\wedge^2(V\otimes \rho)\bigr]^\sym$ of \autoref{lem:repdescription}(1), a generator of this vector space is given by 
\[
\omega=\sum_{i=1}^{n+1} ue_i\wedge ve_i\,;
\]
see \cite[Lemma B.6(1)]{Scala--Coh}. By \autoref{lem:repdescription}(3) and the description \eqref{eq:thetav} of $\vartheta=\vartheta_v$, we have
\begin{equation}\label{eq:doublesum}
\vartheta_*(\omega)=\vartheta\circ \omega=\sum_{i=1}^{n+1}\sum_{j=1}^{n+1}ue_i\wedge ve_i\wedge ve_j\otimes e_j
\end{equation}
in $\Ho^3(G)\cong \Ext^3(\reg_P,\reg_P\otimes \rho)\cong \wedge^3(V\otimes \rho)\otimes \rho$. The base $ue_1,\dots, ue_n,ve_1,\dots,ve_n$ of $V\otimes \rho$ together with the base $e_1,\dots,e_n$ of $\rho$ induce a base of $\wedge^3(V\otimes \rho)\otimes \rho$. We expand \eqref{eq:doublesum} in terms of the induced base, and count the coefficient of $ue_1\wedge ve_1\wedge ve_2\otimes e_2$. Out of the $(n+1)^2$ summands of \eqref{eq:doublesum}, only 4 contribute towards this coefficient, namely:
\begin{itemize}
 \item For $i=1$ and $j=2$, the summand is $ue_1\wedge ve_1\wedge ve_2\otimes e_2$ itself, hence contributing by 1 to the coefficient of this base element in the expansion.
 \item For $i=1$ and $j=n+1$, we use relation \eqref{eq:en+1} to compute
 \[
ue_1\wedge ve_1\wedge ve_{n+1}\otimes e_{n+1}=ue_1\wedge ve_1\wedge(-\sum_{a=1}^n e_a)\otimes (-\sum_{b=1}^n e_b)  
\]
 which contributes to the coefficient of $ue_1\wedge ve_1\wedge ve_2\otimes e_2$ in the base expansion by $(-1)(-1)=1$ via the summand with $a=b=2$.
\item For $i=n+1$ and $j=2$, we compute
 \[
ue_{n+1}\wedge ve_{n+1}\wedge ve_{2}\otimes e_{2}=(-\sum_{a=1}^n ue_a)\wedge (-\sum_{b=1}^n ve_b)\wedge ve_2\otimes e_2 \,. 
 \]
Again, this contributes to the coefficient of $ue_1\wedge ve_1\wedge ve_2\otimes e_2$ in the base expansion by $(-1)(-1)=1$, this time via the summand $a=b=1$. 
\item For $i=n+1$ and $j=n+1$, we compute
 \[
ue_{n+1}\wedge ve_{n+1}\wedge ve_{n+1}\otimes e_{n+1}=(-\sum_{a=1}^n ue_a)\wedge (-\sum_{b=1}^n ve_b)\wedge (-\sum_{c=1}^n ve_c)\otimes (-\sum_{d=1}^n e_d) \,. 
 \]
Among the summands of the fourfold sum on the right-hand side, two contribute to the coefficient of $ue_1\wedge ve_1\wedge ve_2\otimes e_2$, namely for $(a,b,c,d)=(1,1,2,2)$ and $(a,b,c,d)=(1,2,1,2)$. By the skew symmetry between the second and third wedge factor, these two contribution cancel out. Hence, the overall contribution of the  
$i=n+1$ and $j=n+1$ case is $0$.
\end{itemize}
In total, the coefficient of $ue_1\wedge ve_1\wedge ve_2\otimes e_2$ in the base expansion of $\vartheta_*(\omega)$ is $3\neq 0$. In particular, $\vartheta_*(\omega)\neq 0$ showing that $\vartheta_*\colon \Ho^2_{\sym}(\reg_P)\to \Ho^3_{\sym}(G)$ is injective.
\end{proof}

\begin{corollary}\label{cor:hF} We have
 \[
 \vec h_\sym(M)=(0,1,1,1, ?,\dots,?)
 \]   
The question marks indicate that we do not make any claims here on $h^i_\sym(M)$ for $i\ge 4$. 
\end{corollary}

\begin{proof}
We use the long exact sequence associated to the top row of \eqref{eq:Extlattice}:
\[
\dots\to \Ho^i_\sym(G)\to \Ho^i_\sym(M)\to \Ho^i_{\sym}(\reg_P)\to  \Ho^{i+1}_\sym(G)\to \dots\,.
\]
By \eqref{eq:OH} and \eqref{eq:hG}, the start of this sequence splits in the form of
\begin{equation*}
 0\to \Ho_\sym^0(M)\to \CC\xrightarrow{\vartheta_*} \CC^2\to \Ho_\sym^1(M)\to 0\,\,\,\text{and}
\end{equation*}
\begin{equation*}
    0\to \CC\to \Ho_\sym^2(M)\to \CC\xrightarrow{\vartheta_*}\CC^2\to \Ho_\sym^3(M)\to 0\,.
\end{equation*}
By \autoref{lem:thetainj1}, the two displayed maps $\vartheta_*$ are both injective. This implies $h^0_\sym(M)=0$ and $h^1_\sym(M)=h^2_\sym(M)=h^3_\sym(M)=1$, as claimed.
\end{proof}

\begin{remark}
Actually, for all even $i$ with $0\le i\le 2n-2$, where $\Ho^i(\reg_P)$ is one-dimensional, it is true that $\vartheta_*\colon \Ho^i(\reg_P)\to \Ho^{i+1}(G)$ is injective. Hence, we have 
\[
 \vec h_\sym(M)=(0,1,1,1, \dots,1)
 \]
without the question marks in higher degrees. Anyway, since the low degrees are all we need in order to show \autoref{thm:sec2}, we do not proof this claim.
\end{remark}

\begin{lemma}\label{lem:thetainj2}
The map $\vartheta_*\colon \Ext^1_\sym(G,\reg_P)\to \Ext^2_\sym(G,G)$ is injective.
\end{lemma}
\begin{proof}
 By \autoref{lem:hextG}, the vector space $\Ext^1_\sym(G,\reg_P)$ is of dimension two. Hence, it suffices to find two linearly independent vectors in the image of    
\[
\vartheta_*\colon \Ext^1_\sym(G,\reg_P)\cong \bigl[V\otimes \rho\otimes \rho^\vee\bigr]^{\sym} \to \Ext^2_\sym(G,G)\cong \bigl[\wedge^2(V\otimes \rho)\otimes \rho^\vee\otimes\rho\bigr]^{\sym}\,.
\] 
For this purpose, consider the base $e_1', \dots,e_n'$ of $\rho^\vee$ dual to $e_1,\dots ,e_n$ and set \begin{equation}\label{eq:e'relation}
e_{n+1}'=-\sum_{i=1}^ne_i'\,.\end{equation}
Then, the $\sym$-action on $\rho^{\vee}$ is given by $\sigma\cdot e_i'=e_{\sigma(i)}'$. Hence, for $w\in V$, the vector
\[
\phi_w\coloneqq \sum_{i=1}^{n+1}we_i\otimes e_i'\in V\otimes \rho\otimes \rho^\vee
\]
is invariant. We want to show that $\vartheta_*(\phi_v)$ and $\vartheta_*(\phi_u)$ are linearly independent in 
\[
\bigl[\wedge^2(V\otimes \rho)\otimes \rho^\vee\otimes\rho\bigr]^{\sym}\subset \wedge^2(V\otimes \rho)\otimes \rho^\vee\otimes\rho
\]
by comparing their expansion with respect to the base of $\wedge^2(V\otimes \rho)\otimes \rho^\vee\otimes\rho$ which is canonically induced by the bases $V\otimes \rho=\langle ue_1,\dots ,ue_n, ve_1,\dots, ve_n\rangle$, $\rho=\langle e_1,\dots,e_n\rangle$, and $\rho^\vee=\langle e_1',\dots,e_n'\rangle$. 
By \autoref{lem:repdescription}, and recalling the formula \eqref{eq:thetav} for $\vartheta=\vartheta_v$, we have
\begin{align}
 \vartheta_*(\phi_v)&=\sum_{i=1}^{n+1}\sum_{j=1}^{n+1} (ve_i\wedge ve_j)\otimes e_i'\otimes e_j\label{eq:thetaphiv}\\   
\vartheta_*(\phi_u)&=\sum_{i=1}^{n+1}\sum_{j=1}^{n+1} (ue_i\wedge ve_j)\otimes e_i'\otimes e_j\,.\label{eq:thetaphiu}
\end{align}
Obviously, the coefficient of $(ve_1\wedge ve_2)\otimes e_1'\otimes e_2$ in the base expansion of $\vartheta_*(\phi_u)$ is zero, as is the coefficient of $(ue_1\wedge ve_1)\otimes e_1'\otimes e_1$ in the base expansion of $\vartheta_*(\phi_v)$.
If we show that, conversely, the coefficient of $(ve_1\wedge ve_2)\otimes e_1'\otimes e_2$ in $\vartheta_*(\phi_v)$ and the coefficient of $(ue_1\wedge ve_1)\otimes e_1'\otimes e_1$ in $\vartheta_*(\phi_u)$ are non-zero, then we are done.

Only four summands of the double sum \eqref{eq:thetaphiv} contribute towards the coefficient of $(ve_1\wedge ve_2)\otimes e_1'\otimes e_2$ in the base expansion:

\begin{itemize}
\item For $i=1$ and $j=2$, the summand is $(ve_1\wedge ve_2)\otimes e_1'\otimes e_2$ itself. Hence, we get a contribution of 1 to the coefficient.
\item For $i=1$ and $j=n+1$, we use relation \eqref{eq:e'relation} to rewrite the summand as
\[
(ve_1\wedge ve_{n+1})\otimes e_1'\otimes e_{n+1}=ve_1\wedge (\sum_{a=1}^n ve_a)\otimes e_1'\otimes (\sum_{b=1}^n e_b)
\]
which contributes to the coefficient of $(ve_1\wedge ve_2)\otimes e_1'\otimes e_2$ by 1, namely via the summand with $a=2=b$.
\item Analogously, the summand with $i=n+1$ and $j=2$ contributes by 1.
\item For $i=n+1=j$, we rewrite the summand as 
\[
(ve_{n+1}\wedge ve_{n+1})\otimes e_{n+1}'\otimes e_{n+1}=(\sum_{a=1}^n ve_a)\wedge (\sum_{b=1}^n ve_b)\otimes (\sum_{c=1}^n e_c')\otimes (\sum_{d=1}^n e_d)\,.
\]
We get two contributions to the coefficient of $(ve_1\wedge ve_2)\otimes e_1'\otimes e_2$, namely from $(a,b,c,d)=(1,2,1,2)$ and from $(a,b,c,d)=(2,1,1,2)$. However, by the skew symmetry in the first two factors, these cancel each other, and the whole $i=n+1=j$ case does not contribute to the coefficient of $(ve_1\wedge ve_2)\otimes e_1'\otimes e_2$ in the base expansion of \eqref{eq:thetaphiv}.
\end{itemize}
In summary, the coefficient of $(ve_1\wedge ve_2)\otimes e_1'\otimes e_2$ in the base expansion of $\vartheta_*(\phi_v)$ is 3, and not 0.

Let us now compute the coefficient of $(ue_1\wedge ve_1)\otimes e_1'\otimes e_1$ in $\vartheta_*(\phi_u)$ by counting the contributions of the summmands of \eqref{eq:thetaphiu}. 
\begin{itemize}
\item For $i=1$ and $i=1$, the summand contributes 1.
\item For $i=1$ and $j=n+1$, the summand contributes 1.
\item For $i=n+1$ and $j=1$, the summand contributes 1.
\item For $i=n+1=j$, the summand can be rewritten as  
\[
(ue_{n+1}\wedge ve_{n+1})\otimes e_{n+1}'\otimes e_{n+1}=(\sum_{a=1}^n ue_a)\wedge (\sum_{b=1}^n ve_b)\otimes (\sum_{c=1}^n e_c')\otimes (\sum_{d=1}^n e_d)\,.
\]
Now, only $(a,b,c,d)=(1,1,1,1)$ contributes by 1.
\end{itemize}
In summary, the coefficient of $(ue_1\wedge ue_1)\otimes e_1'\otimes e_1$ in the base expansion of $\vartheta_*(\phi_u)$ is 4, and not 0.
\end{proof}

\begin{corollary}\label{cor:alpha*}
The induced maps \[\alpha_*\colon \Hom_\sym(G,G)\to \Hom_\sym(G,M)\quad,\quad \alpha_*\colon \Ext^1_\sym(G,G)\to \Ext^1_\sym(G,M)\] are both isomorphisms. In particular,
\[
\overrightarrow{\ext}_\sym(G,M)=(1,2,?,\dots,?)\,.
\]    
\end{corollary}

\begin{proof}
We consider the long exact sequence associated to the bottom row of \eqref{eq:Extlattice}. By \autoref{lem:hextG}, the start of this sequence splits in the form of
\begin{equation*}
    0\to \Hom_\sym(G,G)\xrightarrow{\alpha_*} \Hom_\sym(G,M)\to 0\,\,\,\text{and}
\end{equation*}
\begin{equation*}
    0\to \Ext^1_\sym(G,G)\xrightarrow{\alpha_*} \Ext^1_\sym(G,M)\xrightarrow{\beta_*} \Ext^1_\sym(G,\reg_P) \xrightarrow{\vartheta_*} \Ext^2_\sym(G,G)
\end{equation*}

The zeros show that the first $\alpha_*$ is an isomorphism and that the second $\alpha_*$ is injective. By \autoref{lem:thetainj2}, the map $\vartheta_*$ is injective, which implies that $\beta_{*}=0$ so the second $\alpha_*$ is also surjective. 
\end{proof}

\begin{lemma}\label{lem:nonzero}
 The map $\vartheta^*\colon \Ext^1_\sym(G,M)\to \Ho^2_\sym(M)$ is non-zero.   
\end{lemma}
\begin{proof}
We consider the commutative diagram
\[
\xymatrix{
&\Ext^1_\sym(G,G)   \ar^{\alpha_*}[r] \ar^{\vartheta^*}[d]  & \Ext^1_\sym(G,M) \ar^{\vartheta^*}[d] \\
\Ho^1_\sym(\reg_P)\ar^{\vartheta_*}[r] &\Ho^2_\sym(G)  \ar^{\alpha_*}[r]  & \Ho^2_\sym(M)        
}
\]
with exact lower row. The upper $\alpha_*$ is an isomorphism by \autoref{cor:alpha*}. We have $\Ho^1_\sym(\reg_P)=0$; see \eqref{eq:OH}. Hence, the lower $\alpha_*$ is injective. 

Accordingly, it suffices to show that $\vartheta^*\colon \Ext^1_{\sym}(G,G)\to \Ho^2_\sym(G)$ is non-zero. 
For this purpose, consider 
\[
\xi\coloneqq \sum_{i=1}^{n+1} ue_i\otimes e_i'\otimes e_i\in \bigl[(V\otimes\rho)\otimes \rho^\vee\otimes \rho \bigr]^\sym\cong \Ext^1_\sym(G,G)\,.
\]
By \autoref{lem:repdescription}, we have 
\begin{equation}\label{eq:xitheta}
\vartheta^*(\xi)=\xi\circ \vartheta= \sum_{i=1}^{n+1}\sum_{j=1}^{n+1}(ue_i\wedge ve_j)\otimes e_i'(e_j)\cdot e_i\in \bigl[\wedge^2(V\otimes \rho)\otimes \rho\bigr]^\sym\cong \Ho^2_\sym(G)\,.  
\end{equation}
Note that, for $i,j\in \{1,\dots,n\}$, we have
\begin{align*}
e_j'(e_i)&=\delta_{i,j} \quad,\quad e_{n+1}'(e_i)=-\sum_{a=1}^n e_a'(e_i)=-1\,,\\ e_j'(e_{n+1})&=e_j'(-\sum_{b=1}^n e_b)=-1\quad,\quad e_{n+1}'(e_{n+1})=(-1)^2\sum_{a=1}^n\sum_{b=1}^ne_a(e_b')=n\,.
\end{align*}
We consider the base expansion of $\vartheta^*(\xi)$ in terms of the base of $\wedge^2(V\otimes \rho)\otimes \rho$ induced by the base $ue_1,\dots, ue_n,ve_1,\dots, ve_n$ of $V\otimes \rho$ and $e_1,\dots, e_n$ of $\rho$. We show that the coefficient of $ue_1\wedge ve_1\wedge e_1$ in that expansion is non-zero, by counting the contributions of the summands of \eqref{eq:xitheta}:
\begin{itemize}
\item For $i=1=j$, the contribution is $1$.
\item For $i=n+1$ and $j=1$, we have the summand
\[
(ue_{n+1}\wedge ve_1)\otimes e_{n+1}'(e_1)\cdot e_{n+1}=(-\sum_{a=1}^n ue_a)\wedge ve_1\otimes (-1)(-\sum_{b=1}^{n} e_b)\,.
\]
This contributes by $(-1)^3=-1$ via the summand $a=b=1$.
\item For $i=1$ and $j=n+1$, we have the summand
\[
(ue_1\wedge ve_{n+1})\otimes e_{1}'(e_{n+1})\cdot e_{1}=ue_1\wedge (-\sum_{a=1}^n ve_a)\otimes (-1)e_1\,.
\]
This contributes by $(-1)^2=1$ via the summand $a=1$.
\item For $i=n+1=j$, we have the summand
\[
(ue_{n+1}\wedge ve_{n+1})\otimes e_{n+1}'(e_{n+1})\cdot e_{n+1}=(-\sum_{a=1}^n ve_a)\wedge (-\sum_{b=1}^n ue_b)\otimes n(-\sum_{c=1}^ne_c)\,.
\]
This contributes by $(-1)^3n=-n$ via the summand $a=b=c=1$.
\end{itemize}
In summary, the coefficient of $ue_1\wedge ve_1\wedge e_1$ in the base expansion of $\vartheta^*(\xi)$ is $1-n$. Remembering the assumption that $n\ge 2$, we have $\vartheta^*(\xi)\neq 0$. 
\end{proof}

\begin{proof}[Finishing the proof of \autoref{thm:sec2}]
We consider the long exact sequence associated to the middle column of \eqref{eq:Extlattice}. For us, the relevant part is

\begin{equation*}
 \begin{tikzcd}
0 \arrow[r]
& \Hom_\sym(M,M) \arrow[r,"\alpha^{*}"]\arrow[d, phantom, ""{coordinate, name=Z}]
& \Hom_\sym(G,M) \arrow[r,"\vartheta^*"]
& \Ho^1_\sym(M) \arrow[dlll,
rounded corners,
to path={ -- ([xshift=2ex]\tikztostart.east)
|- (Z) [near end]\tikztonodes
-| ([xshift=-2ex]\tikztotarget.west)
-- (\tikztotarget)}] \\
\Ext^1_\sym(M,M) \arrow[r]
& \Ext^1_\sym(G,M) \arrow[r,"\vartheta^*"]
& \Ho^2_\sym(M)
\end{tikzcd}   
\end{equation*}
where we already used that $\Ho^0_\sym(M)=0$ by \autoref{cor:hF}. 
By \autoref{cor:alpha*}, we have $\hom_\sym(G,M)=1$, i.e. $\Hom_\sym(G,M)=\CC\langle\alpha\rangle$. 

Hence, the map $\alpha^*\colon \Hom_\sym(M,M)\to \Hom(G,M)$ is surjective. It follows that $\vartheta^*\colon \Hom_\sym(G,M)\to \Ho_\sym^1(M)$ is the zero map. This in turn implies that we have an embedding $\Ho^1_\sym(M)\hookrightarrow\Ext^1_\sym(M,M)$. Recall that $\Ho^1_\sym(M)$ is one-dimensional; see \autoref{cor:hF}. 

By \autoref{cor:alpha*}, \autoref{cor:hF}, and \autoref{lem:nonzero}, the map
\begin{equation*}
   \vartheta^*\colon \Ext^1_\sym(G,M)\cong \CC^2\to\Ho^2_\sym(M)\cong \CC 
\end{equation*}
 is surjective with one-dimensional kernel. In summary, we have a short exact sequence  
\[
0\to \CC\to \Ext^1_\sym(M,M)\to \CC\to 0\,.
\]
Hence, $\ext_\sym^1(M,M)=2$. Since $\Psi$ is an equivalence with $\Psi(M)=F$, we also have
$\ext^1(F,F)=2$.
\end{proof}

\section{Stability of extension bundles}\label{sec:stability}

\subsection{Preliminaries on stability of sheaves}

We quickly recall several two stability notions, both of which will be used in the paper.

Let $(X,H)$ be a polarized smooth projective variety of dimension $n$ and let $F$ be a torsion free coherent sheaf with first Chern class $\operatorname{c}_1(F)\in \NS(X)$ on $X$, then we define the slope of $F$ with respect to $H$ by:
\begin{equation*}
	\mu_H(E)=\frac{\operatorname{c}_1(F)H^{n-1}}{\rk(F)}.
\end{equation*} 
\begin{definition}
Let $H\in \NS(X)$ be an ample class. The sheaf $F$ is slope (semi)stable with respect to $H$ or $H$-slope (semi)stable if for any proper subsheaf $E\subset F$ with rank $0< \rk(E) < \rk(F)$ we have:
\begin{equation*}
	\mu_H(F) \leqp \mu_H(F).
\end{equation*}
\end{definition}

\begin{remark}
	More generally we could allow numerical ample classes in the definition of stability, i.e.\ $H$ living in the ample cone $\Amp(X)$ of $X$. Here we have $\Amp(X)\subset N^1(X)_{\RR}\coloneqq \NS(X)\otimes_{\ZZ}{\RR}$.
\end{remark}

The disadvantage of this definition is that stability with respect $H$ is not linear in $H$. This problem can be fixed by studying stability with respect to movable curve classes; see \cite{greb_movable_2016}.

First recall that a numerical curve class $\gamma\in N_1(X)_{\RR}$ is said to be movable if $\gamma\cdot D\geqslant 0$ for any effective Cartier divisor $D$ under the nondegenerate bilinear pairing
\begin{equation*}
	N_1(X)_{\RR} \times N^1(X)_{\RR}\rightarrow \RR
\end{equation*}
induced from the intersection pairing. We define the slope of $F$ with respect to a movable curve class $\gamma$ by:
\begin{equation*}
	\mu_{\gamma}(F)=\frac{\operatorname{c_1}(F)\gamma}{\rk(F)}.
\end{equation*} 
\begin{definition}
 Let $\gamma\in N_1(X)_{\RR}$ be a movable curve class.	The sheaf $F$ is slope (semi)stable with respect to $\gamma$ or $\gamma$-slope (semi)stable if for any proper subsheaf $E\subset F$ with rank $0< \rk(E) < \rk(F)$ we have:
	\begin{equation*}
		\mu_{\gamma}(E) \leqp \mu_{\gamma}(F).
	\end{equation*}
\end{definition}

\begin{remark}(see \cite[Section 2.3]{greb_movable_2016})
	All elementary properties that are satisfied by coherent sheaves that are slope stable with respect to an ample class also hold when stability is defined by a movable class. 
\end{remark}

\begin{remark}
We consider the map
	\begin{equation*}
		\Phi: N^1(X)_{\RR} \rightarrow N_1(X)_{\RR},\,D\mapsto D^{n-1}\,.
	\end{equation*}
For $H\in \Amp(X)$, the curve class $\Phi(H)\in N_1(X)_{\RR}$ is movable and
\begin{equation*}
	\mu_H(F)=\mu_{\Phi(H)}(F)\,.
\end{equation*}
That is, slope stability with respect to the ample class $H$ is the same as slope stability with respect to the movable curve class $H^{n-1}$.
\end{remark}

From now on we fix the polarized abelian surface $(A,H)$, so $A$ is still an abelian surface and $H\in \NS(A)$ is an ample class. To study slope stability on $\KA$ we recall that there is an isomorphism
\begin{equation}\label{neron-kummer}
	\NS(\KA) \cong \NS(A)_K\oplus \mathbb{Z}\delta.
\end{equation}
\begin{remark}
   The summand $\NS(A)_K\cong \NS(A)$ is constructed as follows: for $D\in \NS(A)$ the class $D_P\coloneqq \sum \overline{p}_i^{*}D\in \NS(\PA)$ is $\sym$-invariant and descends along 
   \begin{equation*}
       \sigma:\PA\rightarrow\PA/\sym=:\SA
   \end{equation*}
   to give a class $D_S\in \NS(\SA)$. Pullback of this class along the Hilbert-Chow morphism $\mathsf{HC}_K: \KA\rightarrow \SA$ finally gives a class $D_K\in\NS(\KA)$. 
\end{remark}

\begin{remark}
The ample class $H\in \NS(A)$ induces the semi-ample (i.e.\ big and nef) class $H_K\in \NS(\KA)$ using the isomorphism \eqref{neron-kummer}. The ample class also induces the ample class $H_P\coloneqq \sum\limits_{i=1}^{n+1} \overline{p}_i^{*}H\in \NS(\PA)$.   
\end{remark}

Using this isomorphism we record the following numerical invariants:
\begin{equation}\label{eq:tautChern}
	\rk(E^{(n)})=(n+1)\rk(E)\,\,\,\text{and}\,\,\,\operatorname{c}_1(E^{(n)})=\operatorname{c}_1(E)_K-\rk(E)\delta.
\end{equation}
(For the first Chern class, see for example \cite[Lemma 1.5]{wandel_kummer} and note that there is an isomorphism $E^{(n)}=\iota^{*}E^{[n+1]}$ where $E^{[n+1]}$ is the induced tautological bundle on the Hilbert scheme $A^{[n+1]}$.)

It follows from \cite[Proposition 2.9]{reede_stable_2022} that if a vector bundle $E\not\cong \cO_A$ is slope stable on $A$ with respect to the ample class $H\in \NS(A)$, then the induced tautological bundle $E^{(n)}$ is slope stable on $\KA$ with respect to an ample class $H_{\epsilon}$. Here $H_{\epsilon}$ is of the form $H_K+\epsilon D\in N^1(\KA)_{\RR}$ for some (in fact any) ample class $D\in \NS(\KA)$ and all $\epsilon$ in a small interval $(0,\epsilon_D)$. In this section we want to study the slope stability of the missing case $\cO_A^{(n)}$ and certain related vector bundles.

\subsection{Stability of the components}

We start with the following negative result:
\begin{lemma}\label{lem:negative-result}
	The tautological bundle $\cO_A^{(n)}$ associated to $\cO_A$ is slope unstable with respect to any ample class $D\in \NS(\KA)$.
\end{lemma}
\begin{proof}
	As the class $\delta$ can be represented by an effective divisor we have $\delta D^{2n-1}>0$ for any ample class $D\in \NS(\KA)$. Hence, $\mu_D(\cO_A^{(n)})<0$ by \eqref{eq:tautChern}. 
	
	By \autoref{thm:sec2} we have $\cO_A^{(n)}\cong \reg_{\KA}\oplus Q$. As $\mu_D(\cO_{\KA})=0$, the summand $\cO_{\KA}$ is a destabilising subbundle of $\cO_A^{(n)}$.
\end{proof}

Although the bundle $\cO_A^{(n)}$ is slope unstable for any ample class $D\in \NS(\KA)$, we will show that in the decomposition
\begin{equation*}
	\cO_A^{(n)}=\cO_{\KA}\oplus Q
\end{equation*}
both summands are slope stable with respect to certain curve classes.

While this result is obvious for the line bundle $\cO_{\KA}$, the proof of slope stability of $Q$ needs some work. 
To help with this task we introduce an adapted version of Stapleton's Hilbert scheme-to-product functor.  
To define the functor, we need the following diagram exhibiting the relations of some relevant schemes:
\begin{equation}
	\label{eqn:Notations}
	\begin{tikzcd}
		P_n(A)_\circ \ar[r, "\sigma_\circ"] \ar[d, hook', "j_P"'] & S_n(A)_\circ \ar[d, hook', "j_S"'] & \KA_\circ \ar[l, "\mathsf{HC}_\circ"'] \ar[d, hook', "j_K"'] \\
		P_n(A) \ar[r, "\sigma_K"] \ar[d, hook', "\tau"'] & S_n(A) \ar[d, hook'] & \KA \ar[l, "\mathsf{HC}_K"'] \ar[d, hook', "\iota"'] \\
		A^{n+1} \ar[r, "\sigma"] & A^{(n+1)} & A^{[n+1]} \ar[l, "\mathsf{HC}"']
	\end{tikzcd}
\end{equation}

Here $\tau$ and $\iota$ are embeddings of a zero fiber of the addition morphism to $A$; $\sigma$ and $\sigma_K$ are quotients by the symmetric group $\sym$; $\mathsf{HC}$ and $\mathsf{HC}_K$ are Hilbert-Chow morphisms. 

Furthermore $j_P$, $j_S$ and $j_K$ are embeddings of open subschemes parametrizing $n+1$ distinct (ordered or unordered) points in $A$. The morphism $\sigma_\circ$ (the restriction of $\sigma$) is a free $\sym$-quotient and $h_\circ$ (the restriction of $\mathsf{HC}_K$) is an isomorphism.

We define the adapted version of Stapleton's functor by:
\begin{equation}\label{def:StapletonFunctor}
	(\_)_P: \Coh(\KA) \rightarrow \Coh_{\sym}(\PA),\,\,\, F\mapsto (j_P)_{*}\sigma_\circ^{*}(\mathsf{HC}_{\circ}^{-1})^{*}j_K^{*}F.
\end{equation}
This defines a left exact functor such that $F_P$ is torsion free (resp. reflexive) if $F$ is. 

\begin{lemma}
We have $Q_P\cong G$ where $G\coloneqq \reg_P\otimes \rho$ as in \autoref{sec:extension}.    
\end{lemma}

\begin{proof}
By \cite[Lemma 1.1]{stapleton_taut_2016} and \cite[Proof of Proposition 1.8]{reede_zhang_22}, we have $(\reg_A^{(n)})_P\cong \reg^{\{n\}}$ where $\reg^{\{n\}}\cong \Ind_{\sym_n}^{\sym_{n+1}} \reg_P\cong \reg_P\otimes R$. 
Furthermore, using that the complement of the open subset $P_n(A)_\circ$ in $P=P_n(A)$ is of codimension 2, we also have $(\reg_{\KA})_P\cong \reg_P$.

The functor $(\_)_P$ is compatible with direct sums and $\Coh_\sym(\PA)$ is Krull-Schmidt. Hence, the assertion follows by comparing the two direct sums $\reg_A^{(n)}\cong \reg_{\KA}\oplus Q$ and \eqref{eq:splitOn}. 
\end{proof}

\begin{remark}
We see that the three sheaves relevant for us, $\reg_A^{(n)}$, $\reg_{\KA}$, and $Q$, have the same images under the two functors $(\_)_P$ and $\Psi^{-1}$ where $\Psi$ is the derived McKay correspondence of \autoref{prop:tautMcKay}. Hence, the reader might wonder why we even introduce both functors. The reason is that they have two different sets of useful properties. The functor $\Psi$ (or $\Psi^{-1}$) is an equivalence, and hence can be used to compute extension groups, as done in \autoref{sec:extension}. On the other hand, $(\_)_P$ preserves subsheaves, ranks, and some compatibility with subsheaves, all of which we use in the proof of \autoref{lem:Q-stable-HK} below. 

Also note that there are many vector bundles on $\KA$ on which the two functor disagree. One such example is $\reg_{\KA}(D)$ where $D=\KA\setminus \KA_\circ$ is the \emph{boundary divisor} parametrizing non-reduced subschemes of $A$.
\end{remark}

In the following, we will often consider the movable curve class
\[
\gamma_0\coloneqq (H_K)^{2n-1}\,.
\]

\begin{lemma}\label{lem:Q-stable-HK}
	With the above notation, $Q$ is slope stable with respect to $\gamma_0$. 
\end{lemma}

\begin{proof}
Let us assume for a contradiction that $Q$ is not stable with respect to $\gamma_0$. Then $Q$ admits a subsheaf $U$ with $\rk(U) < \rk(Q)=n$ and $\mu_{\gamma_0}(U) \geq \mu_{\gamma_0}(Q)$. It transforms to an $\mathfrak{S}_{n+1}$-invariant reflexive subsheaf $U_P$ of $Q_P\cong G$ such that $\rk(U_P) < \rk(G)=n$, and $\mu_{H_P}(U_P) \geq \mu_{H_P}(G)$; for the latter one uses \cite[Equation 2.4]{reede_stable_2022} and the fact that $\mu_{\gamma_0}(\_)=\mu_{H_K}(\_)$. Let $W$ be a Jordan-H\"older factor of $U_P$ of the maximal slope. Note that $W$ does not need to be $\sym_{n+1}$-invariant. Then on the one hand we have
$$ \mu_{H_P}(W) \geq \mu_{H_P}(U_P) \geq \mu_{H_P}(G) = 0\,. $$
On the other hand, we have (as non-equivariant bundles) $Q_P\cong G\cong \reg_{P_n(A)}^{\oplus n}$. As $W\neq 0$, some of the compositions 
\[
W\hookrightarrow Q_P\cong \reg_{P_n(A)}^{\oplus n}\twoheadrightarrow \reg_{P_n(A)}
\]
of the embedding of the subsheaf of $\reg_{P_n(A)}^{\oplus n}$ with the projection to one of the summands must be non-zero. By the stability of $W$ and $\cO_{P_n(A)}$ and the comparison of their slopes, we conclude that  $W \cong \cO_{\PA}$. 
Hence, $\Ho^0(U_P)\neq 0$. As $U_P\subset G$ is equivariant, $\Ho^0(U_P)\subset \Ho^0(G)\cong \rho$ is a subrepresentation. As $\rho$ is irreducible, we get $\Ho^0(U_P)= \Ho^0(G)$. As a non-equivariant sheaf, $G\cong \reg_{\PA}^{\oplus n}$ is generated by global sections. Hence, $\Ho^0(U_P)= \Ho^0(G)$ implies $U_P=G$, contradicting $\rk U_P<\rk G$.
\end{proof}

\subsection{Stability of extension bundles with respect to curve classes}

Now we come to one the main results of this section. As we have seen in \autoref{lem:negative-result}, the bundle $\cO_A^{(n)}$ is slope unstable for any ample class $D\in \NS(\KA)$ while, by \cite[Theorem 2.10]{reede_stable_2022}, all of its deformations $L^{(n)}$ for $[\reg_A]\neq [L]\in\Pic^0(A)$ are.  

On the other hand, in \autoref{lem:Q-stable-HK}, we proved that $Q$ is slope stable with respect to the movable curve class $\gamma_0$. Hence, one may wonder if a non-trivial extension $F_b$, defined by
\begin{equation*}
	\begin{tikzcd}
		0 \arrow[r] & Q \arrow[r] & F_b \arrow[r] & \cO_{\KA} \arrow[r] & 0
	\end{tikzcd}
\end{equation*} 
for $b\in \Ext^1(\cO_{\KA},Q)$ is slope stable with respect to some ample class in $\NS(\KA)$. Such an $F_b$ has the same numerical invariants as $\cO_A^{(n)}$ so that we could replace the latter in a moduli by $F_b$ or even by a family of $\left\lbrace F_b\right\rbrace _{b\in \Ext^1(\cO_{\KA},Q)}$.

\autoref{lem:hG} shows that we have $\ext^1(\cO_{\KA},Q)=2$. Thus there are in fact non-trivial extensions of $Q$ by $\cO_{\KA}$ to work with. We are mainly interested in the isomorphism class of such an extension $F_b$. For these we have the following lemma:

\begin{lemma}\label{lem:P1-family}
For $b,b'\in \Ext^1(\cO_{K_n(A)},Q)$, we have that $F_b\cong F_{b'}$ if and only if $[b]=[b']\in \PP^1\coloneqq \PP(\Ext^1(\cO_{K_n(A)},Q))$. Moreover, there is a $\PP^1$-family of extensions
	\begin{equation}\label{eqn:universal-O-Q}
		0 \rightarrow \pi_1^{*}Q \otimes \pi_2^{*}\cO_{\PP^1}(1) \rightarrow \mathcal{F} \rightarrow \pi_1^{*}\cO_{K_n(A)} \rightarrow 0
	\end{equation}
	on $K_n(A)\times \PP^1$ such that $\mathcal{F}_{[b]}\cong F_b$ for every $b\in \Ext^1(\cO_{K_n(A)},Q)$. 
\end{lemma}

\begin{proof}
For the first assertion, by \cite[Lemma 3.3]{nara_ext}, we only need to check that that $\Hom(Q,\cO_{\KA})=0$ and that both $\cO_{\KA}$ and $Q$ are simple. All of this was already shown in \autoref{lem:hextG} using the derived McKay correspondence.
  
The second part follows from the general results on universal families of extensions of \cite[Section 4]{lange_extensions_1983}.
\end{proof}

\begin{remark}
    From now on, by abuse of notation, we don't distinguish $b$ and its class $[b]$ in $\PP(\Ext^1(\cO_{K_n(A)},Q))$.
\end{remark}

Before we will be able to prove stability of the $F_b$ with respect to an ample class, we will need to prove stability with respect to a certain curve class.

\begin{lemma}\label{cone_int}
	Let $C$ be a convex cone. If $c\in C$ and $d\in C^{\circ}$ then $c+d\in C^{\circ}$.
\end{lemma}

\begin{proof}
	Since $d\in C^{\circ}$ we find $c+d\in c+C^{\circ} \subset C$ as $C$ is closed under addition by convexity. On the other hand $c+C^{\circ}$ is open in $C$ as a translate of an open subset, so $c+C^{\circ}\subset C^{\circ}$ and therefore $c+d\in C^{\circ}$.
\end{proof}

\begin{corollary}\label{cor:big}
	For every $\epsilon >0$ and every $\beta\in \mathrm{Big}(\KA)$, the class $\gamma_{\epsilon}\coloneqq  \gamma_0 + \epsilon \beta$ is big.
\end{corollary}

\begin{proof}
	We have to show that $\gamma_{\epsilon}\in \mathrm{Big}(\KA)=\Mov(\KA)^{\circ}$.
	
	Since $\beta\in \mathrm{Big}(\KA)$ we know that for every $\epsilon >0$ we have $\epsilon\beta\in \mathrm{Big}(\KA)$. We certainly also have $\gamma_0\in \Mov(\KA)$ as $\gamma_0=(H_K)^{2n-1}$ and $H_K\in \NS(\KA)$ is a limit of ample classes. Now apply \autoref{cone_int} to $C=\Mov(\KA)$, $c=\gamma_0$ and $d=\epsilon\beta$. 
\end{proof}

 We also note the following more or less easy observation: 
 
\begin{lemma}\label{lem:exception_Q}
	For each $b \in \PP^1$ and every proper subsheaf $0 \subsetneq E \subsetneq F_b$, we have 
	$$ \mu_{\gamma_0}(E) < \mu_{\gamma_0}(F_b) = 0 \qquad \text{unless} \qquad E = Q. $$
\end{lemma}

\begin{proof}
	Consider the following composition
	\begin{equation*}
		\begin{tikzcd}
			E \ar[r,hook] & F_b \ar[r,two heads] & \cO_{\KA}
		\end{tikzcd}
	\end{equation*}
	and denote its kernel and image as $E_1$ and $E_2$ respectively. Then we obtain the commutative diagram
	\begin{equation*}
		\begin{tikzcd}
			0 \ar[r] & Q \ar[r] & F_b \ar[r] & \cO_{\KA} \ar[r] & 0 \\
			0 \ar[r] & E_1 \ar[r] \ar[u,hook'] & E \ar[r] \ar[u,hook'] & E_2 \ar[r] \ar[u,hook'] & 0.
		\end{tikzcd}
	\end{equation*}
	By \autoref{lem:Q-stable-HK}, we know $Q$ is $H_K$-stable hence $\gamma_0$-stable. It is clear that $\cO_{\KA}$ is also $\gamma_0$-stable. Hence we have
	\begin{align*}
		\mu_{\gamma_0}(E_1) &< \mu_{\gamma_0}(Q) = 0 &\text{ unless } E_1 &= 0 \text{ or } Q, \\
		\mu_{\gamma_0}(E_2) &< \mu_{\gamma_0}(\cO_{\KA}) = 0 &\text{ unless } E_2 &= 0 \text{ or } \cO_{\KA}.
	\end{align*}
	It follows that $$ \mu_{\gamma_0}(E) < 0 \qquad \text{unless} \qquad E = 0, \ Q, \ \cO_{\KA}, \text{ or } F_b. $$
	Since $E$ is a proper subsheaf of $F_b$ and $\mathrm{H}^0(F_b) = 0$ by \autoref{cor:hF}, it follows that the only possibility for $\mu_{\gamma_0}(E)<0$ to fail is $E=Q$.
\end{proof}

\begin{corollary}\label{cor:Fbss}
 For every $b\in \PP^1$, the sheaf $F_b$ is properly $\gamma_0$-semistable.     
\end{corollary}
 
\begin{proof}
 By \autoref{lem:exception_Q}, the only potentially destabilizing subsheaf of $F_b$ is $Q$. As $\mu_{\gamma_0}(F_b)=0$ and $\operatorname{c}_1(Q)=\operatorname{c}_1(F)$, we have $\mu_{\gamma_0}(Q)=\mu_{\gamma_0}(F_b)$.   
\end{proof}

For the following discussion, we need to generalize the Grothendieck's lemma \cite[Theorem 2.29]{greb_movable_2016} to families. Since it might be of some use in other situations, we first state it for an arbitrary smooth projective variety $X$, then apply it in our particular situation in \autoref{lem:finite_S_set}.

\begin{lemma}\label{lem:Grothendieck-general}
    Let $X$ be a smooth projective variety, $\beta$ a big curve class on $X$, and $B$ a quasi-projective scheme. Assume $\cF$ is a vector bundle on $X \times B$. We write $F_b \coloneqq \cF|_{X \times \{b\}}$. Then for every $c\in \RR$
    $$ S \coloneqq \{ \operatorname{c_1}(E) \mid E \subseteq F_b \text{ for some } b \in B \text{ such that } \mu_\beta(E) \geq c \} $$
    is a finite set.
\end{lemma}

\begin{proof}
    The family of vector bundles $\{ F_b \mid b \in B \}$ is bounded, so is the family $\{ F_b^\vee \mid b \in B \}$. By \cite[Lemma 1.7.6]{huybrechts_moduli} there exists some auxiliary vector bundle $V$ on $X$ with surjective maps $V \twoheadrightarrow F_b^\vee$ for all $b \in B$. It follows that $F_b \hookrightarrow V^\vee$ for all $b \in B$, which implies that every subsheaf of $F_b$ for some $b \in B$ is a subsheaf of $V^\vee$. Since $\beta$ is a big class, the absolute version of the Grothendieck lemma \cite[Theorem 2.29]{greb_movable_2016} implies that the set
	$$ \{ \operatorname{c_1}(E) \mid E \subseteq V^\vee \text{ such that } \mu_{\beta}(E) \geq c \} $$
	is finite, and so is the subset $S$.
\end{proof}

From now on, we fix some $\beta\in \mathrm{Big}(\KA)$. For each $\epsilon>0$, we write 
\[
\gamma_\eps\coloneqq \gamma_0+\epsilon \beta
\]
which is big by \autoref{cor:big}.
For each $c \in \RR$, we further denote
$$ S(\epsilon,c) \coloneqq \{ \operatorname{c_1}(E) \mid E \subseteq F_b \text{ for some } b \in \PP^1 \text{ such that } \mu_{\gamma_\epsilon}(E) \geq c \}, $$
then we have

\begin{corollary}\label{lem:finite_S_set}
	For any $c \in \RR$ and $\epsilon > 0$, the set $S(\epsilon,c)$ is finite.
\end{corollary}

\begin{proof}
    It is a special case of \autoref{lem:Grothendieck-general}.
\end{proof}

Note that $\operatorname{c}_1(F_b)$ is independent of $b$. Hence, the same holds for $\mu_{\gamma_\epsilon}(F_b)$ for any fixed $\epsilon>0$, and we can denote
\begin{align*} S(\epsilon) \coloneqq &S\bigl(\epsilon,\mu_{\gamma_\epsilon}(F_b)\bigr)\\= &\{ \operatorname{c}_1(E) \mid E \subseteq F_b \text{ for some } b \in \PP^1 \text{ such that } \mu_{\gamma_\epsilon}(E) \geq \mu_{\gamma_\epsilon}(F_b) \}\,.\end{align*}

\begin{lemma}\label{lem:each_destabilizing_class}
	Assume that $\operatorname{c_1}(E) \in S(\epsilon)$, then there exists some $\epsilon(E)$ with $$0 < \epsilon(E) \leq \epsilon,$$ such that 
	\begin{align*}
		&\operatorname{c_1}(E) \notin S(\theta) \qquad \text{for each } \theta \in (0, \epsilon(E)), \text{ and} \\
		&\operatorname{c_1}(E) \in S(\theta) \qquad \text{for each } \theta \in [\epsilon(E), \infty).
	\end{align*}
	As a consequence, for any $0 < \epsilon_1 < \epsilon_0$, we have $S(\epsilon_1) \subseteq S(\epsilon_0)$.
\end{lemma}

\begin{proof}
	By assumption, we have $$\mu_{\gamma_\epsilon}(E) \ge \mu_{\gamma_\epsilon}(F_b).$$ As $\delta \cdot \beta > 0$ and by the description of the relevant first Chern classes \eqref{eq:tautChern}, we have $$ \mu_{\gamma_\epsilon}(Q) < \mu_{\gamma_\epsilon}(F_b)\,. $$ Hence, we conclude that $E \not\cong Q$. It follows that $$\mu_{\gamma_0}(E) < \mu_{\gamma_0}(F_b)$$ by \autoref{lem:exception_Q}. By definition we have
	$$ \mu_{\gamma_\epsilon} (\_) = \mu_{\gamma_0} (\_) + \epsilon \mu_{\beta} (\_) $$
	which is linear in $\epsilon$. We conclude that
	$$ \mu_{\beta}(E) > \mu_{\beta}(F_b), $$
	and that
	$$ \epsilon(E) = \frac{\mu_{\gamma_0}(F_b) - \mu_{\gamma_0}(E)}{\mu_{\beta}(E) - \mu_{\beta}(F_b)} > 0 $$
	meet the requirement. The second statement follows immediately.
\end{proof}

Now we can prove the following stability result

\begin{proposition}\label{stable_allb}
	There exists some $\epsilon' > 0$, such that for each $\epsilon \in (0, \epsilon')$ and each $b \in \PP^1$, the fiber $F_b$ of the universal extension \eqref{eqn:universal-O-Q} is stable with respect to $\gamma_\epsilon$.
\end{proposition}

\begin{proof}
	We start with an arbitrary $\epsilon_0 > 0$ such that $\gamma_{\epsilon_0}$ is a big curve class. 
	
	If $S(\epsilon_0) \neq \varnothing $, then there exists some $E \subseteq F_b$ such that $\operatorname{c_1}(E) \in S(\epsilon_0)$. Let $\epsilon_1 = \epsilon(E)/2$ where $\epsilon(E)$ is defined in \autoref{lem:each_destabilizing_class}. Since $0 < \epsilon_1 < \epsilon(E) \leq \epsilon_0$, we conclude that $\operatorname{c_1}(E) \notin S(\epsilon_1)$, hence $S(\epsilon_1) \subsetneq S(\epsilon_0)$ by \autoref{lem:each_destabilizing_class} again. 
	
	If $S(\epsilon_1) \neq \varnothing$, we can similarly find $0 < \epsilon_2 < \epsilon_1$ such that $S(\epsilon_2) \subsetneq S(\epsilon_1)$. By \autoref{lem:finite_S_set}, after finitely many similar steps, we can find some $\epsilon_k > 0$ such that $S(\epsilon_k) = \varnothing$. Then we have $S(\epsilon) = \varnothing$ for each $\epsilon \in (0, \epsilon_k)$; in other words, $F_b$ is stable with respect to $\gamma_\epsilon$. Therefore we can set $\epsilon' = \epsilon_k$.
\end{proof}

\subsection{Stability of extension bundles with respect to ample classes}
After having established stability of the $F_b$ with respect to certain curve classes in \autoref{stable_allb},
we now want to show that all $F_b$ for $b\in \PP^1$ are actually slope stable with respect to some ample class. For this we need some preparations:

\begin{lemma}[{Compare \cite[Lemma 4.4]{stapleton_taut_2016}}]\label{lem:gammaadd}
	Let $F$ be a sheaf and $\gamma_1,\gamma_2\in N_1$.
	\begin{enumerate}
		\item
		Let $E\subset F$ be a subsheaf. Then, for every $\gamma=s\gamma_1+t\gamma_2$ with $s\ge 0$ and $t>0$ (hence, in particular, for every $\gamma$ on the line segment connecting $\gamma_1$ and $\gamma_2$), we have:
		\[\mu_{\gamma_1}(E)\le \mu_{\gamma_1}(F)\quad \text{and}\quad \mu_{\gamma_2}(E)< \mu_{\gamma_2}(F)\quad \Longrightarrow \quad \mu_{\gamma}(E)< \mu_{\gamma}(F)\,.\]
		\item 
		Let $F$ be $\gamma_1$-semistable and $\gamma_2$-stable. Then $F$ is also $\gamma$-stable for every $\gamma=s\gamma_1+t\gamma_2$ with $s\ge 0$ and $t>0$ (hence, in particular, for every $\gamma$ on the line segment connecting $\gamma_1$ and $\gamma_2$).
	\end{enumerate}
\end{lemma}
\begin{proof}
	This follows from the fact that $\mu$ is linear in $\gamma$, i.e.
 \begin{equation*}
     \mu_{s\gamma_1+t\gamma_2}(\_)=s\mu_{\gamma_1}(\_) +t\mu_{\gamma_2}(\_).\qedhere
 \end{equation*}
\end{proof}

\begin{proposition}\label{open_stable}
	Let $\alpha\in N_1(\KA)_{\RR}$ be a big class such that $F_b$ is $\alpha$-stable for every $b\in \PP^1$. Then there exists an open neighbourhood $\alpha\in U\subset  N_1(\KA)_{\RR}$ such that for $\gamma\in U$ the sheaves $F_b$ are still $\gamma$-stable for every $b\in \PP^1$.    
\end{proposition}

\begin{proof}
	By \cite[Theorem 3.3]{greb_movable_2016}, for every $b\in \PP^1$, there is an open neighbourhood $\alpha\in U_b$ such that $F_b$ is still $\gamma$-stable for every $\gamma\in U_b$. Note that we cannot conclude our assertion immediately, as $\bigcap_{b\in \PP^1} U_b$ might not be open anymore. However, we can still deduce the result as follows.
	
	Let $e_1,\dots, e_\ell$ be a base of $ N_1(\KA)_{\RR}$. Choose some $\delta_1,\dots,\delta_\ell>0$ such that $\alpha\pm \delta_i e_i$ is still big for every $i=1,\dots,n$. For $i=1,\dots,n$, we consider the two sets
	\[
	S_i^\pm\coloneqq \bigl\{\operatorname{c}_1(E)\mid E\subset F_b \text{ for some $b\in \PP^1$ and } \mu_{\alpha\pm \delta_i e_i}(G)\ge \mu_{\alpha\pm \delta_i e_i}(F_b)\bigr\}  
	\]
	Both sets are finite by the Grothendieck lemma; see \autoref{lem:finite_S_set}. Hence, their union $S_i=S_i^+\cup S_i^-$ is finite too. Let $I\subset \PP^1$ be a finite subset such that every $z\in S_i$ is of the form $z=\operatorname{c}_1(E)$ for some subsheaf $E\subset F_b$ for some $b\in I$. Now, \cite[Theorem 3.3]{greb_movable_2016} gives open neighbourhoods $\alpha\in V_b$ such that $F_b$ is $\gamma$-stable for all $\gamma\in V_b$. Being a finite intersection, $V\coloneqq \bigcap_{b\in I}V_b$ is still an open neighbourhood of $\alpha$. Hence, there is some $0<\eps_i\le \delta_i$ such that $\gamma_i^+\coloneqq \alpha+ \eps_i e_i$ and $\gamma_i^-\coloneqq \alpha- \eps_i e_i$ are both contained in $V$. We claim that $F_b$ is $\gamma_i^+$-stable and $\gamma_i^-$-stable for \emph{all} $b\in \PP^1$. To see this,  let $b\in \PP^1$ and let $E\subset F_b$ be a subsheaf. In the case that $\operatorname{c}_1(E)\notin S_i$, \autoref{lem:gammaadd}(1) gives $\mu_{\gamma_i^+}(E)<\mu_{\gamma_i^+}(F_b)$ and $\mu_{\gamma_i^-}(E)<\mu_{\gamma_i^-}(F_b)$ as $\gamma_i^\pm$ lies on the line segment connecting $\alpha$ and $\alpha\pm \delta_i e_i$. In the case that $\operatorname{c}_1(E)\notin S_i$, there is some $b'\in I$ and some subsheaf $E'\subset F_{b'}$ with $\operatorname{c}_1(E')=\operatorname{c}_1(E)$. Since, by the choices we made, we already know $\gamma_i^\pm$-stability of $F_{b'}$, we get
	\[
	\mu_{\gamma_i^\pm}(E)=\mu_{\gamma_i^\pm}(E')<\mu_{\gamma_i^\pm}(F_{b'})=\mu_{\gamma_i^\pm}(F_b)\,.
	\]
	By \autoref{lem:gammaadd}(2), we get $\gamma$-stability of all $F_b$ for $\gamma$ in the convex hull 
	\[
	\Conv(\gamma_1^+,\gamma_1^-,\dots, \gamma_\ell^+,\gamma_\ell^-)\,.
	\]
	Since $\gamma_i=\alpha\pm \eps_i e_i$ and the $e_i$ form a base, this convex hull contains an open neighbourhood of $\alpha$.
\end{proof}

\begin{lemma}\label{bigclass}
	For any ample class $D\in \NS(\KA)$ the class $DH_K^{2n-2}$ is big. 
\end{lemma}

\begin{proof}
	The Hilbert-Chow morphism $\mathsf{HC}_K:\KA\rightarrow \SA$ is semismall, hence the class $H_K$ is lef and thus by \cite[Theorem 2.3.1]{decataldo_2002} the map
	\begin{equation*}
		q\coloneqq (\_)\cap H_K^{2n-2}: N^1(\KA)_{\mathbb{R}} \rightarrow N_1(\KA)_{\mathbb{R}},\,\,\, B\mapsto BH_K^{2n-2}
	\end{equation*}	
	is a linear isomorphism.
	
	Let $D\in \NS(\KA)\subset N^1(\KA)_{\mathbb{R}}$ be an ample class. We see that have $q(D)=DH_K^{2n-2}\in \Mov(\KA)$. Since the ample cone is open in $N^1(\KA)_{\RR}$ we can pick an open neighborhood $V$ of $D$, that is for every $B\in V$ the class $B$ is ample and so $q(B)=BH_K^{2n-2}\in \Mov(\KA)$. It follows that $q(V)\subset \Mov(\KA)$. But since $q$ is a linear isomorphism and $ N^1(\KA)_{\mathbb{R}}$ is finite dimensional, the set $q(V)$ must in fact be open in $\Mov(\KA)$ that is $q(V)\subset \mathrm{Mov}(\KA)^{\circ}=\mathrm{Big}(\KA)$. We conclude
	\begin{equation*}
		DH_K^{2n-2}=q(D)\in \mathrm{Big}(\KA).\qedhere
	\end{equation*}
\end{proof}

For the next lemma we need some notation: let $P\neq Q$ be two points in $\RR^n$. Denote the connecting vector $\overrightarrow{PQ}$ by $\vec{v}$. Furthermore let $R\coloneqq \mathrm{conv}(P,B_r(Q))$ be the convex hull of $P$ and the (open) ball around $Q$ with some radius $r>0$.

\begin{lemma}\label{curve_cone}
Let $\gamma: [-a,a]\rightarrow \RR^n$ be a differentiable curve such that:
\begin{equation*}
\gamma(0)=P \,\,\,\text{and}\,\,\,\gamma'(0)=\lambda\vec{v}\,\,\,\text{for some $\lambda>0$}.
\end{equation*}
Then there is some $0<a'\leqslant a$ such that $\gamma(t)\in R\setminus \{P\}$ for all $0<t<a'$.
\end{lemma}

\begin{proof}
    The proof of this lemma uses elementary trigonometry and real analysis and is left to the reader.
\end{proof}

We are now ready to state our main result of the section.

\begin{theorem}\label{thm:universal-extension-P1}
	Let $D\in \NS(\KA)$ be an ample class. There is a real number $e>0$ such that for every $\epsilon\in (0,e)$ the vector bundle $F_b$ is slope stable with respect to the ample class $H_{\epsilon}=H_K+\epsilon D$ for any $b\in \PP^1$.
\end{theorem}

\begin{proof}
 By \autoref{bigclass} the class $\beta=DH_K^{2n-2}$ is big. 
 	By \autoref{stable_allb} we know that there is some real number $\epsilon>0$ such that $\gamma_{\epsilon}=\gamma_0+\epsilon \beta \in \bigcap_{b\in \PP^1}\Stab(F_b)$. Using \autoref{open_stable} we can find an open ball neighborhood $W$ of $\gamma_{\epsilon}$ such that 
	\begin{equation*}
		W\subset \bigcap_{b\in \PP^1}\Stab(F_b).
	\end{equation*}
Furthermore, by \autoref{cor:Fbss}, $\gamma_0\in\bigcap_{b\in \PP^1}\SStab(F_b)$. Hence, we can set  
	$R\coloneqq \mathrm{conv}(\gamma_0,W)$ which by \autoref{lem:gammaadd} satisfies 
	\begin{equation*}
		R\setminus \left\lbrace \gamma_0\right\rbrace \subset \bigcap_{b\in \PP^1}\Stab(F_b)\,.
	\end{equation*}
 It is now enough to show that there exists some $e>0$ such that the curve
	\begin{equation*}
		\gamma: [-a,a]\rightarrow N_1(\KA)_{\RR},\,\,\,t\mapsto (H_K+tD)^{2n-1}
	\end{equation*}
	lies in $R\setminus\{\gamma_0\}\subset \bigcap_{b\in \PP^1}\Stab(F_b)$ for all $t$ with $0<t<e$. 
But we have $\gamma(0)=\gamma_0$ and $\gamma'(0)=(2n-1)DH_K^{2n-2}=(2n-1)\beta$. Since $\overrightarrow{\gamma_0 \gamma_{\epsilon}}=\epsilon\beta$, the claim follows from \autoref{curve_cone}.
\end{proof}

\section{Construction of the universal family}\label{sec:component}

In this section, we construct a universal family of stable bundles on $\KA$, and show that it gives a smooth irreducible component of the moduli space.

\subsection{A complete family via elementary modification}

Let $\hA$ be the dual abelian surface of $A$, $o \in \hA$ the neutral point, and $\cL \in \Coh(A \times \hA)$ the universal line bundle. We write
$$ \tau \colon \bA \longrightarrow \hA $$
for the blowup of $\hA$ at $o \in \hA$, and denote the exceptional divisor as $E$. Moreover, for the projections $p$ and $q$ in \eqref{projections_universal}, we define
$\widehat{p} = p \times \id_{\hA}$ and $\widehat{q} = q \times \id_{\hA}$. Then we obtain a universal family of tautological bundles on $\KA$ parametrized by $\hA$
$$ \cL^{(n)} \coloneqq (\widehat{q})_\ast (\widehat{p})^\ast \cL \in \Coh(\KA \times \hA), $$
which induces a family
$$ \cF^{(0)} \coloneqq (\id_{\KA} \times \tau)^\ast \cL^{(n)} \in \Coh(\KA \times \bA) $$
whose restriction on $\KA \times E$ is given as
$$ \cF^{(0)}_E \coloneqq \cF^{(0)}|_{\KA \times E} = \cO_A^{(n)} \boxtimes \cO_E = (\cO_{\KA} \oplus Q) \boxtimes \cO_E. $$
The composition of the obvious surjective maps
$$ \cF^{(0)} \longrightarrow \cF^{(0)}_E \longrightarrow Q \boxtimes \cO_E $$
allows us to construct a new family $\cF^{(1)} \in \Coh(\KA \times \bA)$ as the kernel in the exact sequence
\begin{equation}\label{eqn:langtonF1}
	0 \longrightarrow \cF^{(1)} \longrightarrow \cF^{(0)} \longrightarrow Q \boxtimes \cO_E \longrightarrow 0.
\end{equation}

We view $\cL$ as a family of sheaves parametrized by $\hA$. As a notation, for any morphism $T \to \hA$, we will use $\cL_T$ for the pullback of $\cL$ to $A \times T$. If $T = \{ t \}$ is a single point, we write $\cL_{\{t\}}$ as $\cL_t$ for simplicity. Similarly we can view $\cL^{(n)}$ as a family parametrized by $\hA$, $\cF^{(0)}$ and $\cF^{(1)}$ as families parametrized by $\bA$, so similar notations apply.

The rest of the section will be devoted to show that $\cF^{(1)}$ is a family of stable bundles on $\KA$ parametrized by $\bA$, which identifies $\bA$ as a connected component of the relevant moduli space of stable bundles on $\KA$.

\subsection{Objects in the family}

We first study individual fibers of the family $\cF^{(1)}$. As a preparation we give a more precise description of $\cF^{(1)}_E$. Restricting \eqref{eqn:langtonF1} to $K_n(A) \times E$ and replacing the last surjective map by its kernel we obtain
\begin{equation}\label{eqn:restriction-Langton}
	0 \longrightarrow Q \boxtimes \cO_E(-E) \longrightarrow \cF^{(1)}_E \longrightarrow \cO_{K_n(A)} \boxtimes \cO_E \longrightarrow 0.
\end{equation}
Further restricting \eqref{eqn:restriction-Langton} to a closed point $y \in E$, we observe that $\cF^{(1)}_y$ fits in an exact sequence
\begin{equation}\label{eqn:exceptional-fiber}
	0 \longrightarrow Q \longrightarrow \cF^{(1)}_y \longrightarrow \cO_{K_n(A)} \longrightarrow 0. 
\end{equation}
The following observation will be crucial to our construction.
	
\begin{lemma}\label{lem:non-trivial-ext}
	The extension \eqref{eqn:exceptional-fiber} is non-trivial for each closed point $y \in E$.
\end{lemma}

\begin{proof}
	We fix a smooth curve $\overline{C} \subseteq \bA$ which intersects $E$ transversely at $y$. By restricting \eqref{eqn:langtonF1} to $\KA \times \overline{C}$ we obtain
	\begin{equation}\label{eqn:first_on_C}
		0 \longrightarrow \cF^{(1)}_{\overline{C}} \longrightarrow \cF^{(0)}_{\overline{C}} \longrightarrow Q \boxtimes \cO_y \longrightarrow 0,
	\end{equation}
	which is an elementary transformation of $\cF^{(0)}_{\overline{C}}$. 
	
	If $\cF^{(1)}_y = Q \oplus \cO_{\KA}$, then we can similarly apply another elementary transformation and obtain
	\begin{equation*}\label{eqn:second_on_C}
		0 \longrightarrow \cF^{(2)}_{\overline{C}} \longrightarrow \cF^{(1)}_{\overline{C}} \longrightarrow Q \boxtimes \cO_y \longrightarrow 0.
	\end{equation*}
	Combining both sequences we obtain
	\begin{equation*}\label{eqn:third_on_C}
		0 \longrightarrow \cF^{(2)}_{\overline{C}} \longrightarrow \cF^{(0)}_{\overline{C}} \longrightarrow \overline{Q}^{(2)} \longrightarrow 0,
	\end{equation*}
	where $\overline{Q}^{(2)}$ fits in an exact sequence
	\begin{equation}\label{eqn:Q_2}
		0 \longrightarrow Q \boxtimes \cO_y \longrightarrow \overline{Q}^{(2)} \longrightarrow Q \boxtimes \cO_y \longrightarrow 0.
	\end{equation}
	Since $\overline{Q}^{(2)}_y$ is a quotient of $\cF^{(0)}_y = \cO_{\KA} \oplus Q$, and $\Hom(\cO_{\KA}, Q) = 0$, it follows that $\overline{Q}^{(2)}_y = Q \boxtimes \cO_y$. Therefore \eqref{eqn:Q_2} implies that $\overline{Q}^{(2)}$ is flat over $\overline{D}$, where $\overline{D}$ is the first order thickening of $y$ in $\overline{C}$. We restrict $\cF^{(0)}_{\overline{C}}$ to $\overline{D}$ and obtain
	\begin{equation}\label{eqn:G2_Q2}
		0 \longrightarrow P^{(2)} \longrightarrow \cF^{(0)}_{\overline{D}} \longrightarrow \overline{Q}^{(2)} \longrightarrow 0,
	\end{equation}
	where the kernel $P^{(2)}$ is also flat over $\overline{D}$ with closed fiber $P^{(2)}_y = \cO_{\KA}$. Since 
	$$ \Ext^1(\cO_{\KA}, \cO_{\KA}) = 0, $$
	it follows from \cite[Chapter 1, Theorem 2.7]{hartshorne_deformation_2010} that
	\begin{equation}\label{eqn:P2_trivial}
		P^{(2)} = \cO_{\KA} \otimes \cO_{\overline{D}}.
	\end{equation}
	Let $D = \tau(\overline{D}) \subseteq A$, then $\overline{D} \cong D$ via $\tau$, with the closed point of $D$ being $o \in \hA$. The following exact sequence follows from \eqref{eqn:G2_Q2} and \eqref{eqn:P2_trivial}
	\begin{equation}\label{eqn:extension_on_D}
		0 \longrightarrow \cO_{\KA} \otimes \cO_D \longrightarrow \cL^{(n)}_D \longrightarrow Q^{(2)} \longrightarrow 0\,.
	\end{equation}
In particular,
\[
\Ho^0(\KA\times D,\cL_D^{(n)})\supset \Ho^0(\KA\times D,\cO_{\KA} \otimes \cO_D )\cong \Ho^0(D,\reg_D)\cong \CC^2\,.
\]
This contradicts \autoref{lem:LDcoh} below.
 Therefore our assumption $\cF^{(1)}_y = Q \oplus \cO_{\KA}$ cannot hold, which concludes the proof.
\end{proof}

The following lemma was used in the proof of the above result.

\begin{lemma}\label{lem:LDcoh}
 With the above notations, we have $\Ho^0(\KA\times D,\cL_D^{(n)})\cong \CC$.   
\end{lemma}

\begin{proof}
Let $\pr^K\colon \KA\times D\to \KA$ and $\pr^A\colon A\times D\to A$ the projections to the first factors. By flat base change, we have $\pr^K_*\cL_D^{(n)}\cong (\pr^A_*\cL_D)^{(n)}$. To shorten the notation, let us write $E\coloneqq \pr^A_*\cL_D$. Then 
\[
\Ho^0(\KA\times D,\cL_D^{(n)})\cong \Ho^0(\KA, \pr^K_*\cL_D^{(n)})\cong \Ho^0(\KA,E^{(n)})\cong \Ho^0(A, E)
\]
where the last isomorphism is the degree zero part of \eqref{eq:tautH}. By the work of Mukai \cite{Muk}, the Fourier--Mukai transform $\Phi\coloneqq \FM_{\cL}\colon D^b(\hA)\to D^b(A)$ is an equivalence. We have $\Phi(\reg_D)\cong E$ and $\Phi(\reg_o)\cong \reg_A$. Hence
\[
\Ho^0(A, E)\cong \Hom_A(\reg_A,E)\cong \Hom_A(\Phi(\reg_o),\Phi(\reg_D))\cong \Hom_{\hA}(\reg_o, \reg_D)\cong \CC\,.\qedhere
\]
\end{proof}

We can now summarize the property of individual fibers of $\cF^{(1)}$.

\begin{proposition}\label{prop:fibers_of_F1}
	$\cF^{(1)} \in \Coh(\KA \times \bA)$ is a vector bundle. For every ample class $D \in \NS(\KA)$, there exists some $e>0$, such that for each closed point $y \in \bA$, the fiber $\cF^{(1)}_y$ is a stable vector bundle on $\KA$ with respect to $H_\epsilon = H_K + \epsilon D$ for each $\epsilon \in (0, e)$. Moreover the fibers of $\cF^{(1)}$ are pairwise non-isomorphic.
\end{proposition}

\begin{proof}
	We proceed in several steps.
	
	\textsc{Step 1.} We claim that $\cF^{(1)}_E$ is the universal extension \eqref{eqn:universal-O-Q} of $\cO_{K_n(A)}$ by $Q$. Moreover, there exists some $e>0$, such that the fibers $\cF^{(1)}_y$ for $y \in E$ are pairwise non-isomorphic $H_\epsilon$-stable bundles for each $\epsilon \in (0, e)$.
	
	Indeed, \autoref{lem:non-trivial-ext} shows that the fiber of $\cF^{(1)}$ over each point of $E$ is a non-trivial extension. Due to the universality of \eqref{eqn:universal-O-Q} we obtain a classifying morphism $$\varphi \colon E \longrightarrow \PP (\Ext^1(\cO_{K_n(A)}, Q)),$$
	which is a morphism from $\PP^1$ to $\PP^1$. We denote the degree of $\varphi$ by $d$, then by pulling back \eqref{eqn:universal-O-Q} we obtain
	\begin{equation*}
		0 \longrightarrow Q \boxtimes \cO_E(d) \longrightarrow \cF^{(1)}_E \longrightarrow \cO_{K_n(A)} \boxtimes \cO_E \longrightarrow 0.
	\end{equation*}
	Comparing it with \eqref{eqn:restriction-Langton} and noticing that $E^2 = -1$ we conclude that $d=1$, hence $\varphi$ is an isomorphism. The existence of $e$ follows immediately from \autoref{thm:universal-extension-P1}.
	
	\textsc{Step 2.} We claim that the fibers $\cF^{(1)}_y$ for $y \in \bA \setminus E$ are pairwise non-isomorphic vector bundles. Moreover, they are $H_\epsilon$-stable for each $\epsilon \in (0, e)$ possibly after shrinking $e$ from Step 1.
	
	Indeed, away from the exceptional divisor $E \subseteq \bA$, we have
	\begin{equation}\label{eqn:B-minus-E}
		\cF^{(1)}_{\bA \setminus E} = \cF^{(0)}_{\bA \setminus E} = \cL^{(n)}_{\hA \setminus \{o\}}.
	\end{equation}
	Since fibers of $\cL^{(n)}$ are distinct tautological bundles, the claim follows from \cite[Theorem 2.10]{reede_stable_2022}.
	
	\textsc{Step 3.} It remains to show that fibers of $\cF^{(1)}_{\bA \setminus E}$ and fibers of $\cF^{(1)}_E$ are never isomorphic.
	
	Indeed, by \cite[Proposition 2.8]{reede_stable_2022}, we know that fibers of $\cF^{(1)}_{\bA \setminus E}$ are $\gamma_0$-stable. However, by \autoref{cor:Fbss}, we find that fibers of $\cF^{(1)}_E$ are strictly $\gamma_0$-semistable. Therefore they can never be isomorphic.
\end{proof}

\subsection{Component of the moduli space}

In this section, we show that the family $\cF^{(1)}$ gives rise to a smooth connected component of the moduli space of $H_\epsilon$-stable sheaves on $\KA$.

By \autoref{prop:fibers_of_F1}, $\cF^{(1)}$ is a family of $H_\epsilon$-stable bundles parametrized by $\bA$, which induces a classifying morphism 
$$ f \colon \bA \longrightarrow \cM, $$
where $\cM$ is the moduli space of $H_\epsilon$-stable sheaves on $\KA$.

\begin{theorem}\label{prop:main-result}
	 The classifying morphism $f$ embeds $\bA$ as a smooth connected component of the moduli space $\cM$.
\end{theorem}

\begin{proof}
	By \autoref{prop:fibers_of_F1}, we see that $f$ is injective on closed points. For each closed point $y \in \bA$, we have that
	$$ \dim T_{f(y)}\cM = \dim \Ext^1(\cF^{(1)}_y, \cF^{(1)}_y) = 2, $$
	which follows from \autoref{prop:Mea}(12) for $y \in \bA \setminus E$, and from \autoref{thm:sec2} for $y \in E$. Since $\bA$ is a smooth surface, it follows from \cite[Lemma 1.6]{reede_examples_2021} that $f$ is an isomorphism from $\bA$ to a connected component of $\cM$. 
\end{proof}


\end{document}